\documentclass[12pt,oneside,reqno]{amsart}
\usepackage{amsmath,amssymb, amsthm,mathtools}
\usepackage[margin=1in]{geometry}
\usepackage[parfill]{parskip}
\usepackage{enumitem}
\usepackage{tikz}
\usepackage{tikz-cd} 
\usepackage{dsfont}
\usepackage{mathrsfs}
\usepackage{bbm}
\usepackage{dirtytalk}
\usepackage{bm}
\usepackage{microtype}
\usepackage{caption}
\usepackage{subcaption}
\usepackage{array}
\usepackage{longtable}

\providecommand{\N}{\mathbb{N}}

\providecommand{\R}{\mathbb{R}}
\providecommand{\C}{\mathbb{C}}

\providecommand{\Z}{\mathbb{Z}}

\DeclareMathOperator{\Tr}{Tr}

\usepackage{centernot}

\renewcommand{\vec}[1]{\boldsymbol{#1}}

\newcommand{\paren}[1]{\left( #1 \right)}
\newcommand{\brac}[1]{\left[ #1 \right]}

\newcommand{\abs}[1]{\left\vert#1\right\vert}

\newcommand{\set}[1]{\left\{#1\right\}}

\DeclareMathOperator{\im}{im}

\newtheorem{Theorem}{Theorem}
\newtheorem{Lemma}[Theorem]{Lemma}
\newtheorem{Definition}[Theorem]{Definition}
\newtheorem{Proposition}[Theorem]{Proposition}

\newtheorem{Corollary}[Theorem]{Corollary}
\newtheorem{Conjecture}[Theorem]{Conjecture}

\newcounter{cnstcnt}

\usepackage{graphicx}
\begin{document}
\title[PLGT Correlation Length]{A Topological Formula for \\ Potts Lattice Gauge Theory Correlations}
\author{Paul Duncan}
\email{pauldunc@iu.edu}
\address{Department of Mathematics, Indiana University, Bloomington, IN 47408, USA}
\author{Benjamin Schweinhart}
\email{bschwei@gmu.edu}
\address{Department of Mathematical Sciences, George Mason University, Fairfax, VA 22030, USA}

\begin{abstract}
    We exhibit a formula relating the correlation between Wilson loop variables in Potts lattice gauge theory to a topological quantity in the plaquette random cluster model.
As applications we prove exponential decay of correlations between slowly growing Wilson loop variables for Ising lattice gauge theory on $\Z^3$ at all but the critical temperature, we demonstrate that the correlation length is finite at sufficiently high or low temperatures in any dimension, and  we show that the correlation length of the model on $\Z^4$ with free boundary conditions equals that of the dual model with constant boundary conditions when $q$ is prime.
\end{abstract}
\maketitle

\section{Introduction}
Lattice gauge theories were introduced by Wilson~\cite{wilson1974confinement} as discretized versions of Yang--Mills theory, with the special case of Ising lattice gauge theory being defined earlier by  Wegner~\cite{wegner1971duality}.
These models assign group-valued spins to the edges of a cell complex $X$ according to the Wilson action, which is defined in terms of the product of spins on the boundary of each two-dimensional face. Here we study a generalization of Ising lattice gauge theory called Potts lattice gauge theory (PLGT)~\cite{kogut1980z}, 
where spins in the multiplicative group of complex $q$-th roots of unity $\Z\paren{q}$ are placed on the edges, with energy determined by whether the product of spins around each face is one. That is, for a function $f$ assigning spins to the edges of a finite complex $X,$ define 
$$H\paren{f}=-\sum_{\sigma\in X^{(2)}}I_{\delta f\paren{\sigma}=1}\,$$
where $\sigma$ ranges over the $2$-cells (often called plaquettes) of $X$ and  $\delta f\paren{\sigma}$ is the (oriented) product of the spins assigned to the edges incident to $\sigma.$ This can be extended to infinite complexes via limits on finite subcomplexes, as was done with free and wired boundary conditions in Proposition 18 and Proposition 68 of~\cite{duncan2025sharp}, respectively. PLGT is motivated both because it coincides with Euclidean lattice gauge theory when $q=2$ or $q=3$ and because it is a natural analogue of the classical Potts model which assigns spins to vertices and has an action on the edges. Ising lattice gauge theory has also attracted interested due to its relationship with Kitaev's toric code~\cite{agrawal2025geometric, dennis2002topological}

The most important observable for the Potts model is the spin correlation function $\tau\paren{v_i,v_j}.$ Due to the gauge symmetry of Potts lattice gauge theory, there is no information in the correlations between spins on individual edges. As such, the observables of interest involve Wilson loop variables $W_\gamma$, the product of the spins around loops $\gamma$. There is a growing literature on the asymptotics of the expectations of Wilson loop variables in Ising/Potts lattice gauge theory~\cite{chatterjee2020wilson,laanait1989discontinuity,duncan2025sharp,aizenman2025geometric,forsstrom2025free} as part of the larger program on Euclidean lattice gauge theories; see~\cite{chatterjee2016yang} for a survey and also~\cite{cao2020wilson,garban2023improved,forsstrom2022decay,cao2025random} for a selection of more recent work. This can be thought of as a generalization of $\tau\paren{v_i,v_j},$ because the Wilson loop variable involves the product of spins around the boundary of a two-dimensional rectangle and $\tau\paren{v_i,v_j}$ is the product of the spins on the boundary of a ``one-dimensional rectangle''. In addition, when spins are taken in a complex matrix group, their product along a path (a Wilson line variable) is a discrete analogue of a connection and a Wilson loop variable measures the path dependence of this discrete connection.
A second natural generalization of the spin correlation function is the correlation between two Wilson loop variables. The minimal example of such an observable is found by taking the loops to be the boundaries of two individual two-dimensional faces at varying distances. This suggests two natural quantities of interest: the asymptotic properties of Wilson loop variables as the size of the loop is taken to $\infty$ and the asymptotic correlation of two Wilson loop variables of a fixed size as they are taken further apart. The latter is expected to exhibit exponential decay except perhaps at the critical point of the former. We state two conjectures for the behavior of PLGT. We note that they are natural analogues of conjectures for Euclidean lattice gauge theory; see~\cite{chatterjee2016yang}.

Let $\nu=\nu\paren{\beta,q,d}$ be Potts lattice gauge theory on $\Z^d$ defined as a weak limit of finite volume measures. On $\Z^4$ there is a natural duality transformation which  maps PLGT with free boundary conditions to PLGT with wired boundary conditions, with the self-dual point being $\beta_{sd}(q,4)=\log\paren{1+\sqrt{q}}.$ This duality was first observed by Wegner~\cite{wegner1971duality} and Kogut~\cite{kogut1980z} on the level of partition functions. The corresponding duality for the plaquette random cluster model is defined on the level of states; see Corollary 24 of~\cite{duncan2025homological}. 

\begin{Conjecture}[Area Law/Perimeter Law]
  \label{conj:sharpnessPLG}
  Let $q\geq 2$ and $d\geq 3$ and consider Potts lattice gauge theory $\nu\paren{\beta,q,d}$ on $\Z^d.$  
There exists a constant $0<\beta_c\paren{q,d}<\infty$ and constants $0<a\paren{\beta,q},b\paren{\beta,q}<\infty$ so that for rectangular boundaries  $\gamma$ in $\Z^d,$
\[
\nu_{\beta,q,d}(W_\gamma)\sim \begin{cases}  \exp(-a(\beta,q) \mathrm{Area}(\gamma)) &\qquad \beta <\beta_{c}(q,d)\\    \exp(- b\paren{\beta,q}\mathrm{Per}(\gamma)) &\qquad \beta >\beta_{c}(q,d)\end{cases}\,.
\]
as the dimensions of $\gamma$ are taken to $\infty$ suitably. In addition, $\beta_{c}(q,4)=\beta_{sd}(q,4)=\log\paren{1+\sqrt{q}}$ is the self-dual point of PLGT.
\end{Conjecture}
In previous work~\cite{duncan2025sharp}, the authors proved the first statement of Conjecture~\ref{conj:sharpnessPLG} for $d=3$ conditional on a conjecture of Pisztora~\cite{pisztora1996surface} on the regularity of the random cluster model in slabs. The proof is unconditional for $q=2$ since it is known that the critical temperature of the Ising model on $\Z^d,$ $d\geq 3$ is also the threshold for the surface tension of that model~\cite{lebowitz1981surface}. See~\cite{aizenman2025geometric} for an alternative proof in that case which uses the random current expansion. 

To state the second conjecture, let $\sigma_N$ be the plaquette  $\brac{0,1}^2\times\set{N}\times\set{0}^{d-3}.$ For a plaquette $\sigma$ let $W_{\sigma}=W_{\partial\sigma}$ be the associated Wilson loop variable. Define the \textbf{correlation length} $\xi_{\beta}=\xi_{\beta,q,d}$ via the limit
 $$-\frac{1}{\xi_{\beta}}=\lim_{N\to\infty} \frac{\log\paren{\mathrm{Cov}_{\nu}\paren{W_{\sigma_1},W_{\sigma_N}^{-1}}}}{N}=\lim_{N\to\infty} \frac{\log\paren{\mathrm{Cov}_{\nu}\paren{W_{\sigma_1},\overline{W_{\sigma_N}}}}}{N}\,.$$  The existence of this limit is a consequence of reflection positivity, via an argument nearly identical\footnote{PLGT is reflection positive by classical arguments~\cite{osterwalder1978gauge,seiler1982gauge}: see Lemma~\ref{lemma:RP} in the appendix for a proof. Then by Equation A.9 of~\cite{borgs1996covariance, grimmett2006random} we can write $\mathrm{Cov}_{\nu}\paren{W_{\sigma_1},\overline{W_{\sigma_N}}}$ in the form $\mathrm{Cov}\paren{A,T^N A}$ for a related observable $A$ where $T$ is a translation operator. The rest of the argument follows identically.} to that of Theorem 4.1 in~\cite{borgs1996covariance}. The usual proof of the corresponding statement for the Potts model using subadditivity and the random cluster representation does not obviously generalize. This will become evident when we state the topological formula for Wilson loop correlations below. The term $W_{\sigma_N}^{-1}$ is inverted because it is what appears in the reflection positive pairing.

\begin{Conjecture}[Mass Gap]\label{conj:massgapPLGT}
Assume the same hypotheses as the previous conjecture, and let $\beta_c\paren{q,d}$ be the same inverse temperature appearing therein. Then $0<\xi_{\beta,q,d}<\infty$ for all $\beta\neq \beta_c\paren{q,d}.$ 
\end{Conjecture}
The correlation length between spins on vertices in the Potts model has been of significant interest~\cite{duminil2020exponential,duminil2021discontinuity,duminil2017continuity, borgs1996covariance}, though a direct analogue of Conjecture~\ref{conj:massgapPLGT} would instead consider ``energy-energy correlations'' between spin differences on distant edges\footnote{For a plaquette $\sigma$ we can write $W_{\sigma}\paren{f}=\delta f\paren{\sigma}.$ Similarly, if $g\in C^0\paren{X;\Z\paren{q}}$ is sampled from the Potts model and $e$ is an edge we can write the spin difference across $e$ as $\delta g\paren{e}.$}. In fact, it is a special case of Theorem~\ref{thm:massgapdual} below that (for prime $q$) the correlation length of PLGT in three dimensions equals the correlation length for energy-energy correlations in the dual Potts model. In analogy with the planar Potts model, one might guess that the limit $\lim_{\beta\nearrow\beta_c\paren{q,d}} \xi_{\beta}$ is infinite for small $q$ and finite for large $q.$ In three dimensions, non-rigorous methods suggest that $\xi_{\beta_c}=\infty$ for $q=2$ and $0<\xi_{\beta_c}<\infty$ for $q\geq 3.$ Specifically, the Potts model on $\Z^3$ is expected to exhibit a first-order phase transition for $q\geq 3$ and have a finite correlation length at $\beta_c$~\cite{janke1997three, yang2025microscopic}. On the other hand, the Ising model on $\Z^3$ is thought to undergo a second-order phase transition and a renormalization group argument predicts the scaling relation
$$\Delta_{\epsilon}=3-\frac{1}{\nu}$$
where $\Delta_{\epsilon}$ is the scaling exponent for the energy-energy correlation and $\nu$ is the correlation length critical exponent~\cite{cardy1996scaling}. The conformal bootstrap predicts that $\nu\approx0.629971$~\cite{kos2016precision}. Applying this in conjunction with Proposition~\ref{prop:correlation_alternate} below and duality (Theorem~\ref{thm:duality}) yields that  
$$\mathrm{Cov}_{\nu}\paren{W_{\sigma_1},W_{\sigma_N}^{-1}}\propto N^{-2\Delta_{\epsilon}}$$ 
for Ising lattice gauge theory on $\Z^3$ where  $\Delta_{\epsilon}\approx 1.4126.$

We study the behavior of the correlation length through the plaquette random cluster model (PRCM). The PRCM~\cite{hiraoka2016tutte,duncan2025sharp,shklarov} is a cellular representation of PLGT, generalizing the relationship between the classical random cluster model and the Potts model. Specifically, the two-dimensional PRCM with free boundary conditions on a finite cell complex $X$ with parameters $p\in\brac{0,1}$ and $q\in\N_{\geq 2}$ is the random complex $P$ containing all vertices and edges of $X$ and a random subset of the $2$-cells with distribution
$$\mu^{\mathbf{f}}_{X}\paren{P} = \mu^{\mathbf{f}}_{X,p,q}\paren{P} \propto p^{\abs{P}}\paren{1-p}^{\abs{X^{\paren{2}}} - \abs{P}}\abs{H^1\paren{P;\Z_q}}$$
where $\abs{X^{\paren{2}}}$ and $\abs{P}$ denote the number of plaquettes in $X$ and $P,$ respectively, and $H^1\paren{P;\Z_q}$ is the first cohomology group of $P$ with coefficients in the group $\Z_q$ of integers modulo $q.$ Roughly speaking, $\abs{H^1\paren{P;\Z_q}}$ counts the number of ``linearly independent surfaces'' in $P.$ See Section~\ref{sec:overview} for a rigorous definition. For a box $\Lambda\subset \Z^d$, the PRCM with wired boundary conditions on $\Lambda$ is the measure $\mu^{\mathbf{w}}_{\Lambda}\paren{P}$ obtained by replacing $H^1\paren{P;\Z_q}$ with $H^1\paren{P\cup \partial \Lambda;\Z_q}.$  

The PRCM $P$ and PLGT $f$ with free boundary conditions are coupled (see Section 5 of~\cite{duncan2025sharp}) so that the conditional distribution of $P$ given $f$ is independent plaquette percolation on the plaquettes satisfying $\delta f=1$ and the conditional distribution of $f$ given $P$ is the uniform distribution on the $1$-dimensional cocycles of $P$: $\set{f:\delta f\paren{\sigma}=1\forall \sigma \in P}.$ The same relationship holds between the PRCM with wired boundary conditions and PLGT with constant or wired boundary conditions. 

 This coupling generalizes to $i$-dimensional Potts (hyper)lattice gauge theory $\nu_{X,\beta,q,d,i},$ which assigns spins to the $i$-cells of a cell complex, and its cellular representation by the $(i+1)$-dimensional PRCM $\mu_{X,p,q,i+1}.$ The $(i+1)$-dimensional PRCM on $X$ is a random percolation subcomplex of $X$: a subcomplex that contains all cells of dimensions $0\leq j\leq i$ and some of the cells of dimension $i+1.$  We defer the definitions of these models to Section~\ref{sec:overview} but state all of our results for them since they hold with identical proofs.
 These more general notions will be useful at some points in proofs relevant to the case $i=1,$ for example when showing finiteness of the correlation length in the low temperature regime. In order to state our results we will also need notation for PLGT with various boundary conditions. For a box $\Lambda,$ denote by $\nu^{\mathbf{f}}_{\Lambda,\beta,q,d,i}=\nu_{\Lambda,\beta,q,d,i}$ PLGT with free boundary conditions, $\nu^{1}_{\Lambda,\beta,q,d,i}$ PLGT with constant boundary conditions, and $\nu^{\mathbf{w}}_{\Lambda,\beta,q,d,i}$ PLGT with wired boundary conditions. Constant boundary conditions impose the requirement that $ f\mid_{\partial \Lambda}=1.$ Gauge invariant observables have the same expectations with respect to constant and wired boundary conditions, so we state our results for wired boundary conditions for notational simplicity.

Using the coupling between PLGT and the PRCM, we can convert statistics on observables in the former to topological questions in the latter. For simplicity, we restrict consideration to the PRCM on a general box $\Lambda,$ on the cube $\Lambda_N \coloneqq\brac{-N,N}^d,$ or on all of $\Z^d.$ 

We now define the basic filling event, which generalizes to higher dimensions the event that two points are connected by a path. Roughly speaking, $V_{\gamma}$ is the event that $\gamma$ is the boundary of a surface of plaquettes. 

\begin{Definition}
Let $P$ be an $(i+1)$-dimensional percolation subcomplex of $X.$ For $\gamma\in Z_i\paren{X;\Z_q}$ let $V_{\gamma}=V^{\mathbf{f}}_{\gamma}\paren{q}$ be the event that $\brac{\gamma}=0$ in $H_i\paren{P;\Z_q}.$ Similarly, set  $V^{\mathbf{w}}_{\gamma}=V^{\mathbf{w}}_{\gamma}\paren{q}$ to be the event that  $\brac{\gamma}=0$ in $H_i\paren{P\cup \partial \Lambda;\Z_q}$ if $X=\Lambda$ and to be the event that $\brac{\gamma}=0$ in $H_i\paren{P\cup \partial \Lambda_N;\Z_q}$ for all sufficiently large $N$ if $X=\Z^d.$ 
\end{Definition}

The topological terms used above are explained in greater detail in Section~\ref{sec:overview}. It turns out that the Wilson loop expectations in PLGT is the same as the probability of $V_{\gamma}$ in the PRCM with appropriate parameters.

\begin{Theorem}[\cite{duncan2025topological,duncan2025sharp,shklarov}]
\label{thm:comparisongeneral}
Let $0\leq i \leq d-2,$ let $X = \Z^d$ or let $X = \Lambda \subset \Z^d$ be a box, let $\gamma$ be an $i$-cycle in $X,$ let $q\geq 2,$ let $\#\in\set{\mathbf{f},\mathbf{w}},$ and let $\nu=\nu_{X,\beta,q,i}^{\#}$ be PLGT and $\mu=\mu^{\#}_{X,1-e^{-\beta},\Z_q,i+1}$ be the corresponding PRCM.
Then
\[\mathbb{E}_{\nu}\paren{W_{\gamma}}=\mu\paren{V_{\gamma}^\#}\,,\]
\end{Theorem}

Our first main result expands the dictionary between PLGT and the PRCM to cover correlations between pairs of loops in addition to expectations of single loops. In order to state it, we introduce another topological event. 

\begin{Definition}
Let $P$ be an $(i+1)$-dimensional percolation subcomplex of $X$. For $\gamma,\gamma'\in Z_i\paren{X;\Z_q}$ define $V_{\gamma,\gamma'}$ to be the event that $\brac{\gamma}=\brac{\gamma'}\neq 0$ in $H_i\paren{P;\Z_q}$. Similarly, let $V_{\gamma,\gamma'}^{\mathbf{w}}$ be the event that $\brac{\gamma}=\brac{\gamma'}\neq 0$ in $H_{i}\paren{P\cup \partial \Lambda;\Z_q}$ if $X=\Lambda$ and the event that $\brac{\gamma}=\brac{\gamma'}\neq 0$ in $H_{i}\paren{P\cup \partial \Lambda_N;\Z_q}$ for sufficiently large $N$ if $X=\Z^d.$  
\end{Definition}
Informally, when $i=1$ $V_{\gamma,\gamma'}$ is the event that two loops bound a surface in $P$ but neither bounds one alone. See Figure~\ref{fig:vdoublegamma}.
This event was previously defined in Section 3 of \cite{abraham1984random}, in the special case of $2$-dimensional Bernoulli plaquette percolation ($q=1,i+1=2$) where homology is taken with coefficients in $\Z_2.$ They write that it is ``a model of a correlation function'' and study a special case where there chain connecting $\gamma$ and $\gamma'$ is required to be an SOS tube.  Our next result makes rigorous the connection between $V_{\gamma,\gamma'}$ and correlation functions in lattice gauge theories.  

\begin{figure}[ht]
\center
\includegraphics[width=0.6\textwidth]{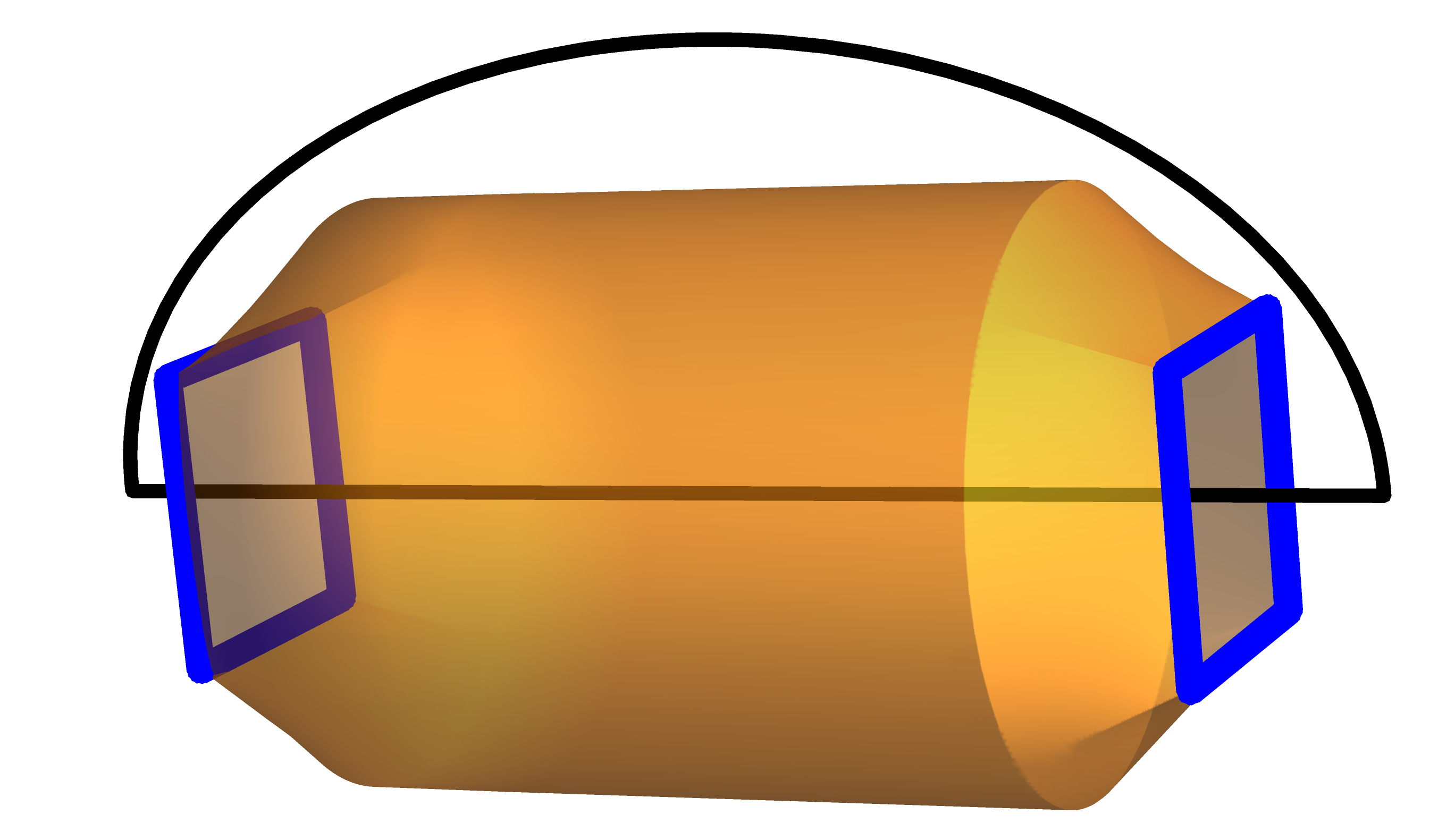}
\caption{An illustration of the event $V_{\gamma,\gamma'}.$ $\gamma$ and $\gamma'$ are thick blue squares and the plaquette surface connecting them is depicted in orange. The events $V_{\gamma}$ and $V_{\gamma'}$ are precluded by the open dual path shown as black.} 
\label{fig:vdoublegamma}
\end{figure}

\begin{Theorem}\label{thm:correlation}
Under the same hypotheses as Theorem~\ref{thm:comparisongeneral}, 
    $$\mathrm{Cov}_{\nu}\paren{W_{\gamma},W_{\gamma'}^{-1}}= \mu\paren{V_{\gamma,\gamma'}^{\#}}+\mathrm{Cov}_{\mu}\paren{V^{\#}_{\gamma},V^{\#}_{\gamma'}}\,.$$

\end{Theorem}

Theorem~\ref{thm:correlation} is a direct analogue of the fact that truncated spin correlations in the Potts model coincide with truncated connection probabilities in the RCM, which follows immediately from the Edwards-Sokal coupling~\cite{borgs1996covariance}. As applications, we prove the following two results. The correlation length $\xi^{\#}_{\beta,q,d,i}$ for $i$-dimensional PLGT with $\#$ boundary conditions is defined analogously to the case $i=1.$

\begin{Theorem}\label{thm:highlowtemp}
Let $0\leq i \leq d-2,$ $q\geq 2,$ $d \geq 3,$ and $\#\in\set{\mathbf{f},\mathbf{w}}.$ Then $0<\xi^{\#}_{\beta,q,d,i}<\infty$ when $\beta$ is sufficiently large or small. 
\end{Theorem}

\begin{Theorem}\label{thm:codimone}
Let $\beta_{c}(q)$ be the critical inverse temperature for the Potts model on $\Z^d$ and  $\beta^*\paren{\beta}=\log\paren{\frac{e^{\beta}+q-1}{e^{\beta}-1}}\,.$ Then
 $0<\xi_{\beta,2,d,d-2}<\infty$ when $\beta\neq \beta^*\paren{\beta_c\paren{2}},$ and $0<\xi_{\beta,q,d,d-2}<\infty$ when $\beta>\beta^*\paren{\beta_c\paren{q}}.$
\end{Theorem}
Furthermore, in proving Theorem~\ref{thm:codimone} we prove a stronger result which allows the sizes of the loops in the definition of the correlation length to grow at a rate of $o\paren{N^{\frac{1}{d-1}}}.$  The boundary conditions are irrelevant in this case because there is a unique Gibbs measure for each value of $\beta$ for the $q=2$ random cluster model on $\Z^d$~\cite{bodineau2006translation,raoufi2020translation,aizenman2015random}, and uniqueness is known in the subcritical regime for general $q$~\cite{grimmett1995stochastic, aizenman1988discontinuity}. As we will see in the proofs, we are able to prove more in the case $q=2$ because of better tools available to analyze the dual Ising model. However, we expect the difference to be a limitation of our methods, not a difference of behavior in the supercritical regime across different values of $q.$

Finally, in Theorem~\ref{thm:massgapdual}, we show that Wilson loop correlation lengths for the boundaries of single plaquettes are self-dual for prime $q,$ when appropriately defined. Towards that end, let $\tau_N=\set{0}^{d-i-1}\times\brac{N,N+1}\times\brac{0,1}^{i}$ and define $\hat{\xi}^{\#}_{\beta,q,d,i}$ via the limit
$$-\frac{1}{\hat{\xi}^{\#}_{\beta,q,d,i}}=\lim_{N\to\infty} \frac{\log\paren{\mathrm{Cov}_{\nu}\paren{W_{\tau_0},W_{\tau_N}}}}{N}\,.$$
We will establish the existence of this limit in the proof of the next theorem.
The term $W_{\tau_N}$ is not inverted because the reflection about $x_{d-i}=(N+1)/2$ reverses the orientation of $\tau_0,$ so the reflection positive pairing is between $W_{\tau_0}$ and $W_{\tau_N}.$  

\begin{Theorem}\label{thm:massgapdual}If $q$ is prime and $0\leq i\leq d-2$ then $$\xi^{\mathbf{f}}_{\beta,q,d,i}=\hat{\xi}^{\mathbf{w}}_{\beta^*\paren{\beta,q},q,d,d-i-2}$$ where 
$\beta^*\paren{\beta,q}=\log\paren{\frac{e^{\beta}+q-1}{e^{\beta}-1}}\,.$
\end{Theorem}

\begin{figure}[b]
\centering
\includegraphics[height=0.30\textwidth]{%
			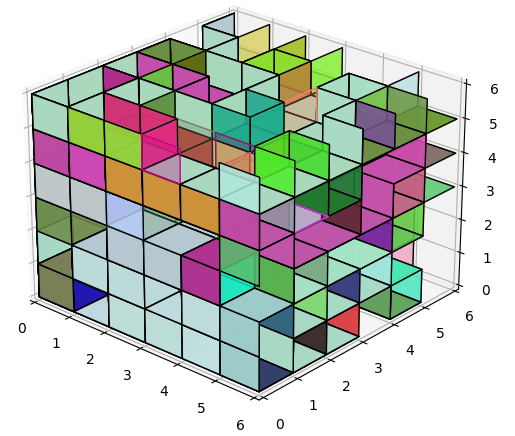}
\caption{An example of a two-dimensional percolation subcomplex of a box in $\Z^3.$}
\label{fig:plaquette}           
\end{figure}

We now give a brief outline of the paper. In Section~\ref{sec:overview} we review the definitions of homology and cohomology, PLGT, and the PRCM. Next, Section~\ref{sec:preliminaries} establishes a number of basic topological results. We prove Theorem~\ref{thm:correlation} in Section~\ref{sec:formulas}, as well as exhibiting a simpler analogue of that theorem when $\gamma$ and $\gamma'$ are the boundaries of single plaquettes and $q=2.$  We then apply the results of Section~\ref{sec:formulas} in Section~\ref{sec:applications} to prove Theorems~\ref{thm:highlowtemp}, ~\ref{thm:codimone}, and~\ref{thm:massgapdual}. We prove the case of Theorem~\ref{thm:massgapdual} where $q$ is an odd prime in an appendix.

\section{Background and Previous Work}
\label{sec:overview}
We begin by introducing some notation and conventions. The cell complexes that we work with are all cubical complexes: complexes that are collections of (hyper)cubes of various dimensions  that are closed under inclusion, meaning that if the complex contains a cell, it also contains all faces of that cell. For example, a two-dimensional cubical complex is a collection of vertices, edges, and faces so that the boundary of an edge is two vertices and the boundary of a face is a finite number of edges. We will mostly be interested in the cubical complex $\Z^d$ (or a finite subcomplex thereof) formed by subdividing $\R^d$ into unit hypercubes with vertices at the integer lattice points. Recall that an $i$-dimensional percolation subcomplex of $X$ is a subcomplex that contains all cells of dimensions $0\leq j\leq i-1$ and some of the cells of dimension $i.$  A random $2$-dimensional percolation subcomplex of $\Z^3$ is shown in Figure~\ref{fig:plaquette}. The vertices and edges of $\Z^d$ constitute the familiar nearest-neighbor graph on $\Z^d.$ The faces are unit squares (called plaquettes) which are obtained from $\brac{0,1}^2\times\set{0}^{d-2}$ by translations and rotations. We will assume that the edges and faces of the cubical complex are oriented: for example, the edges $\paren{v,w}$ and $\paren{w,v}$ correspond to the same line segment with opposite orientations. The vertices, edges, and faces are called the $0$-cells, $1$-cells, and $2$-cells. We denote the collection of $i$-cells of a subcomplex $X$ by $X^{\paren{i}}.$ In this article the ambient complex will either be $\Z^d$ or a finite box of the form $\Lambda = \prod_{i=1}^d \brac{a_i,b_i},$ where $a_i,b_i \in \Z$ for each $i.$ We also introduce notation for the special case of cubes, writing $\Lambda_N \coloneqq \brac{-N,N}^d.$ 

The dual cubical complex $\paren{\Z^d}^{\bullet} \coloneqq \Z^d + \paren{1/2,\ldots,1/2}$ will also be relevant for our purposes. Each $i$-dimensional cell $\sigma$ of $\Z^d$ intersects exactly one $(d-i)$-dimensional cell of $\paren{\Z^d}^{\bullet},$ which we denote as $\sigma^{\bullet}.$ Then if $P$ is an $i$-dimensional percolation subcomplex of $\Z^d,$ we define the dual complex $P^{\bullet}$ to be the $(d-i)$-dimensional percolation subcomplex of $\paren{\Z^d}^{\bullet}$ satisfying 
\[\sigma^{\bullet} \in P^{\bullet} \iff \sigma \notin P\]
for each $\sigma \in \paren{\Z^d}^{\paren{i}}.$ For general subcomplexes $P,$ we define $P^{\bullet}$ to be the subcomplex of $\paren{\Z^d}^{\bullet}$ which contains precisely the $j$-cells $\sigma^{\bullet}$ for all $\sigma\in \Z^d$ so that $\sigma\notin P.$ We also define the dual of a box $\Lambda = \prod_{i=1}^d \brac{a_i,b_i}$ to be $\Lambda^{\bullet} = \prod_{i=1}^d \brac{a_i-1/2,b_i+1/2},$ the smallest dual box containing $\Lambda.$

\begin{figure}[ht]
\center
\includegraphics[width=0.26\textwidth]{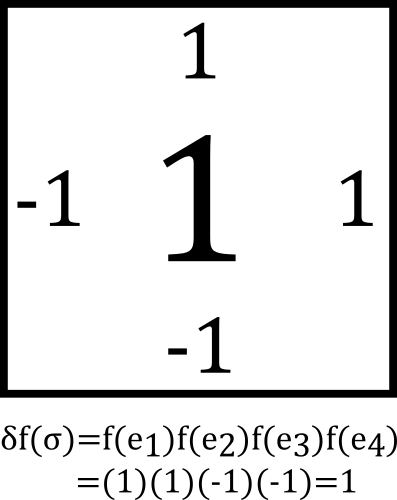}
\caption{An illustration of the coboundary operator $\delta:C^1\paren{X;\Z\paren{2}}\to C^2\paren{X;\Z\paren{2}}.$}
\label{fig:coboundary}
\end{figure}

Functions assigning group elements to the faces of a cell complex are basic objects of interest in both statistical mechanics and algebraic topology but different notational conventions are used. In statistical mechanics, spins are usually taken to be elements of a multiplicative group $G$, such as $\Z\paren{q}$ or a complex matrix group. This naturally allows the definition of quantities like Wilson loop expectations.  In algebraic topology, these groups are usually taken in an additive group such as $\paren{\Z_q,+}$ or $\paren{\Z,+}.$ This allows for the convenient use of the ring structure of $\Z$ and $\Z_q.$ For either type of group $G$ we denote the group of functions assigning spins to the  $i$-cells of $X$ by $C^i\paren{X;G}.$ Elements of $C^i\paren{X;G}$ are referred to either as $i$-cochains or discrete $i$-forms. We write $-\sigma$ to denote the $i$-cell $\sigma$ with reversed orientation and require that any $i$-cochain $f$ satisfies $f\paren{-\sigma}=f\paren{\sigma}^{-1}.$ Note that $C^i\paren{X;\Z_q}$ and $C^i\paren{X;\Z\paren{q}}$ are isomorphic as groups, and cohomology groups defined using either notational convention will be isomorphic. 

Let $G$ be a multiplicative group. The \textbf{coboundary} or discrete exterior derivative is the group homomorphism $\delta^i:C^i\paren{X;G}\to C^{i+1}\paren{X;G}$ defined by setting $\delta^i f\paren{\sigma}$ to be the product of $f\paren{\tau}$ over the oriented $i$-cells $\tau$ in the boundary of an $(i+1)$-cell $\sigma,$ and extending multiplicatively (when the group operation is written as addition, we instead extend linearly). For example, if $e=\paren{v,w}$ is an oriented edge and $f\in C^0\paren{X;G}$ then $\delta f\paren{e}=f\paren{w}f\paren{v}^{-1}.$ See Figure~\ref{fig:coboundary}. We will sometimes drop the superscript when the dimension is understood. The coboundary operator is well-defined when $G$ is an abelian group, but only for $i=0$ when $G$ is non-abelian since there is no canonical ordering on the boundary faces. However, if $G$ is a complex matrix group and $f\in C^1\paren{X;G},$ then $\Tr\delta f$ is well-defined, since the trace of a product of matrices is cyclically invariant. When $G$ is an additive group, the coboundary operator is defined analogously.

\looseness=-1 

The \textbf{cohomology groups} are defined by 
\[H^i\paren{X;G} \coloneqq \ker \delta^i / \im \delta^{i-1}\,.\]

We primarily use these groups in the definition of the PRCM, although they also appear in calculations in Section~\ref{sec:preliminaries}. The key events in the PRCM are defined using the dual notion of homology. We will only define these using additive notation to match the standard convention in algebraic topology, so we now let $G$ be an additive abelian group. The chain group $C_i\paren{X;G}$ consists of formal linear sums of oriented $i$-cells in $X$ with coefficients in $G.$ With the appropriate conversion from the multiplicative group $\Z\paren{q}$ to the additive group $\Z_q$, the cochains defined earlier can be thought of as linear functions on chains which map $\sigma$ to the spin assigned to $\sigma.$ The \textbf{boundary} map $\partial_i : C_i\paren{X;G} \to C_{i-1}\paren{X;G}$ is a linear map that sends an $i$-dimensional cube to a sum of its $(i-1)$-dimensional faces, oriented so that $\partial_i \circ \partial_{i+1} = 0.$  For example, when $G = \Z_2,$ a chain $\alpha$ can be identified with a subset of the $i$-cells (the ones with nonzero coefficients). The boundary of a chain is then the set of $(i-1)$-cells contained in an odd number of the cells of $\alpha.$ In the case that each $(i-1)$-cell is contained in at most 2 $i$-cells of $\alpha,$ $\partial_i \alpha$ is exactly the boundary of $\alpha$ in the usual point-set topological sense. As such, the boundary of the $2$-cell in Figure~\ref{fig:coboundary} is the sum of the four incident edges. 

The \textbf{homology groups} are then defined by

\[H_i\paren{X;G} \coloneqq \ker \partial_i / \im \partial_{i+1}\,.\]

Recall that for $\gamma \in C_i\paren{X;G}$ with $\partial \gamma = 0,$ $V_{\gamma}$ is the event that $0 = \brac{\gamma} \in H_i\paren{P;G},$ or in words that there is a sum of plaquettes whose boundary is $\gamma.$ Recall also that for $\gamma,\gamma' \in C_i\paren{X;G}$ with $\partial \gamma = \partial \gamma' = 0,$ $V_{\gamma,\gamma'}$ is the event that $0 \neq \brac{\gamma} = \brac{\gamma'} \in H_i\paren{P;G},$ or in words that there is a sum of plaquettes whose boundary is $\gamma - \gamma',$ but that neither $\gamma$ nor $\gamma'$ is the boundary of a sum of plaquettes individually. Since our loops are often boundaries of single plaquettes, in order to simplify notation we will write $V_{\sigma} \coloneqq V_{\partial \sigma}$ and $V_{\sigma,\sigma'} \coloneqq V_{\partial \sigma,\partial \sigma'}.$

\subsection{Potts Lattice Gauge Theory}

\begin{Definition}
Let $\Lambda \subset \Z^d$ be a box, let $q\geq 2,$ and let $0\leq i\leq d-1.$ For $f\in C^{i}\paren{\Lambda;\Z\paren{q}}$ define
\begin{equation*}
    H\paren{f}=-\sum_{\sigma\in \Lambda^{\paren{i+1}}}I_{\delta f\paren{\sigma}=1}\,.
\end{equation*}
The $i$-dimensional $q$-state Potts (hyper)lattice gauge theory on $\Lambda$ with inverse temperature $\beta$ is the measure
\[\nu_\Lambda\paren{f}=\nu^{\mathbf{f}}_{\Lambda,\beta,q,i}\paren{f} \propto e^{-\beta H\paren{f}}\,.\]
Furthermore, let $f' \in C^{i}\paren{\Lambda;\Z\paren{q}}$ satisfy $f'\paren{\sigma} = 1$ for $\sigma\in \paren{\partial \Lambda}^{\paren{i}}$ and $f'\paren{\sigma} = f\paren{\sigma}$ for $\sigma \in \paren{\Lambda \setminus \partial \Lambda}^{\paren{i}}.$ Then the $q$-state Potts lattice gauge theory with constant boundary conditions and inverse temperature $\beta$ is the measure $\nu^1_\Lambda=\nu^1_{\Lambda,\beta,q,i}$ defined by
\[\nu^1_\Lambda\paren{f}\propto \nu_{\Lambda}\paren{f'} \propto e^{-\beta H\paren{f'}}\,.\]
\end{Definition}
In addition to $\Z\paren{q}$ symmetry, PLGT with free boundary conditions has what is called a gauge symmetry: for a spin assignment $g\in C^{i-1}\paren{\Lambda;\Z\paren{q}}$ the transformation $f\mapsto f\,\delta g$ leaves the Hamiltonian unchanged. Note that we could replace constant (1) boundary conditions with any coboundary on $\partial \Lambda$ --- or by imposing wired boundary conditions that just require that $\delta f\mid_{\partial \Lambda}=1$ ---without changing the behavior of gauge-invariant observables. Denote PLGT on $\Lambda$ with wired boundary conditions by $\nu_\Lambda^{\mathbf{w}}.$ 

The infinite volume measures $\nu_{\Z^d},\nu_{\Z^d}^{\mathbf{w}},$ and $\nu^1_{\Z^d}$ are defined via weak limits over increasing cubes. The existence of these limits can be shown using the PRCM, as is done in~\cite{duncan2025sharp}. 

\begin{Definition}
The \textbf{Wilson loop variable} for an oriented $i$-cycle $\gamma$ and $f\in C^{i}\paren{X;\Z\paren{q}}$ is
\[W_{\gamma}\paren{f}=f\paren{\gamma}\,\,.\]
 That is, $W_{\gamma}\paren{f}$ is the oriented product of the spins assigned to the $i$-cells of $\gamma.$ 
\end{Definition}

\subsection{The Plaquette Random Cluster Model}
The $i$-dimensional plaquette random cluster model is a cellular representation of $(i-1)$-dimensional Potts lattice gauge theory. 
\begin{Definition}
The $i$-dimensional \textbf{plaquette random cluster model (PRCM)} with free boundary conditions on a finite cell complex $X$ with parameters $p\in\brac{0,1}$ and $q\in\N_{\geq 2}$ is the random $i$-dimensional percolation subcomplex $P$ with the following distribution:
$$\mu^{\mathbf{f}}_{X}\paren{P} = \mu^{\mathbf{f}}_{X,p,q}\paren{P} \propto p^{\abs{P}}\paren{1-p}^{\abs{X^{\paren{i}}} - \abs{P}}\abs{H^{i-1}\paren{P;\Z_q}}$$
where $\abs{X^{\paren{i}}}$ and $\abs{P}$ denote the number of $i$-faces of $X$ and $P.$

Moreover, if $\Lambda=\brac{-N,N}^d\subset \Z^d$ is a box, the PRCM with wired boundary conditions on $\Lambda$ is the measure
$$\mu^{\mathbf{w}}_{\Lambda}\paren{P} = \mu^{\mathbf{w}}_{\Lambda,p,q}\paren{P} \propto p^{\abs{P}}\paren{1-p}^{\abs{\Lambda^{\paren{i}}} - \abs{P}}\abs{H^{i-1}\paren{P\cup \partial \Lambda;\Z_q}}\,.$$
The weak limits of these measures as $\Lambda\nearrow \Z^d$ exist by standard stochastic monotonicity arguments and are denoted by $\mu^{\mathbf{f}}_{\Z^d,p,q}$ and $\mu^{\mathbf{w}}_{\Z^d,p,q}.$ 
\end{Definition}

The following statements are reproduced from~\cite{duncan2025sharp}, with minor modifications.

\begin{Proposition}\label{prop:couple}
Let $\Lambda\subset \Z^d$ be a box and let $q\in\N+1,\,$$\beta \in [0,\infty),$ and $p = 1-e^{-\beta}.$ Consider the couplings on $C^{i}\paren{\Lambda;\Z\paren{q}}\times \set{0,1}^{\Lambda^{\paren{i+1}}}$ by
\[K^{\mathbf{f}}\paren{f,P} \propto \prod_{\sigma \in \Lambda^{\paren{i+1}}}\brac{\paren{1-p}I_{\set{\sigma \notin P}} + p I_{\set{\sigma \in P,\delta f\paren{\sigma}=1}}}\]
and
\[K^{\mathbf{w}}\paren{f,P} \propto I_{\delta f\mid_{\partial \Lambda} =1}\prod_{\sigma \in \Lambda^{\paren{i+1}}}\brac{\paren{1-p}I_{\set{\sigma \notin P}} + p I_{\set{\sigma \in P,\delta f\paren{\sigma}=1}}}\,.\]

Then for $\#\in\set{\mathbf{f},\mathbf{w}},$ 
\begin{itemize}
    \item The first marginal of $K^{\#}$ is
    $\nu^{\#}_{\Lambda,\beta,q,i}.$ 
    \item The second marginal of $K^{\#}$ is
    $\mu^{\#}_{\Lambda,p,\Z_q,i+1}.$ 
\end{itemize}
\end{Proposition}

\begin{Theorem}\label{thm:duality}
Let $\Lambda\subset \Z^d$ be a box, let $q\geq 2,$ and let $p\in\brac{0,1}.$ 
Then, if $P$ is an $i$-dimensional percolation subcomplex of $\Lambda$
\begin{equation*}
   \mu_{\Lambda,p,q,i}^{\mathbf{f}}\paren{P} = \mu_{\Lambda^{\bullet},p^*\paren{p},q,d-i}^{\mathbf{w}}\paren{P^{\bullet}}
\end{equation*}
where
$$p^*\paren{p}=\frac{\paren{1-p}q}{\paren{1-p}q + p}\,.$$
\end{Theorem}

\section{Topological Preliminaries}
\label{sec:preliminaries}

We start by describing a few special cases in which the homology $H_k\paren{P;\Z_q}$ and cohomology $H^k\paren{P;\Z_q}$ are both isomorphic to $\paren{\Z_q}^{\beta_k\paren{P;\Z_q}},$ where $\beta_k\paren{P;\Z_q}$ is called the $k$-th Betti number. In these cases, we will be able to give a direct relationship between the event $V_{\sigma}$ and the effect of adding/removing $\sigma$ from $P$ on the Betti numbers of $P.$ This allows us to relate $V_{\sigma}$ to the event $V_{\sigma^{\bullet}}$ in the dual.

When $q$ is prime, $H_k\paren{X;\Z_q}$ is a vector space over the field $\Z_q$ and we define the $k$-th Betti number $\beta_k\paren{X,q}$ to be its dimension. We extend the definition of the Betti numbers to non-prime $q$ for certain values of $k.$ When $k=0,$ $H_0\paren{X;\Z_q}\cong \paren{\Z_q}^{\kappa\paren{X}}$ where $\kappa\paren{X}$ is the number of components of $X,$ and when $X$ is a bounded subcomplex of $\Z^d$ we have that $H_{d-1}\paren{X;\Z_q}\cong \paren{\Z_q}^{\kappa\paren{\R^d\setminus X}-1}.$ The latter is a consequence of Alexander duality, which we state below. As such, we set $\beta_0\paren{X;q}=\kappa\paren{X}$ and  $\beta_{d-1}\paren{X;q}=\kappa\paren{\R^d\setminus X}-1.$ Moreover, if $P$ is an $\paren{d-1}$-dimensional percolation subcomplex we also have that  
$$H_{d-2}\paren{P;\Z_q}\cong \Z_q^{\beta_{{d-2}}\paren{P,2}}$$
and we set $\beta_{d-2}\paren{P,q}=\beta_{{d-2}}\paren{P,2}$ in that case as well. See Proposition 59 in~\cite{duncan2025sharp}. 

In these special cases, the cohomology is isomorphic to the homology and is determined by the Betti numbers. This follows from the universal coefficient theorem for cohomology (Theorem 3.2 in~\cite{hatcher2002algebraic}), except the latter case which is Proposition 59 in~\cite{duncan2025sharp}. 

\begin{Lemma}
Let $X$ be a bounded subcomplex of $\Z^d.$ If $q$ is prime, $k=0,$ $k=d-1,$ or $k=d-2$ and $X$ is a $\paren{d-1}$-dimensional percolation subcomplex of $\Z^d$ then 
$$H_k\paren{X;\Z_q} \cong H^k\paren{X;\Z_q} \cong \paren{\Z_q}^{\beta_k\paren{X,q}}.$$  
\end{Lemma}

We now give an extension to higher dimensions of the fact that an extra edge added to a graph either joins two different components or creates a loop. Since the boundary maps are between adjacent dimensions, the addition of an $i$-cell to $P$ affects only $\beta_i\paren{P}$ and $\beta_{i-1}\paren{P},$ and in particular can only increase the former and can only decrease the latter. It follows from the rank-nullity theorem that the addition of an $i$-cell cannot increase $\beta_i\paren{P}$ by more than one or decrease $\beta_{i-1}\paren{P}$ by more than one. When $q$ is prime, the Euler-Poincar\'{e} formula (see Theorem 2.44 of~\cite{hatcher2002algebraic}) tells us that adding a single $i$-cell increases the quantity $\beta_i\paren{P}-\beta_{i-1}\paren{P}$ by exactly one. The preceding paragraph proves the following lemma, which is essentially the same as the one given in~\cite{delfinado1993incremental}:

\begin{Lemma}\label{lemma:addone}
Let $P$ be an $i$-dimensional percolation subcomplex of $\Lambda\subset \Z^d$ and let $\sigma$ be an $i$-cell that is not contained in $P.$ If $q$ is prime, $i=1,$ or $i=d-1,$ then either
$$\beta_{i-1}\paren{P\cup \sigma}=\beta_{i-1}\paren{P}-1, \beta_i\paren{P\cup \sigma}=\beta_i\paren{P}, P\notin V_{\sigma}$$
or
$$\beta_{i-1}\paren{P\cup \sigma}=\beta_{i-1}\paren{P}, \beta_i\paren{P\cup \sigma}=\beta_i\paren{P}+1, P\in V_{\sigma}\,.$$
\end{Lemma}

Alexander duality relates the homology of a subset of Euclidean space to the cohomology of its complement. We give an elementary statement and proof of this result for subcomplexes of $\Z^d.$ Define the discrete Hodge star operator $\ast:C^i\paren{\paren{\Z^d}^{\bullet};\Z_q}\to C_{d-i}\paren{\Z^d;\Z_q}$ by setting $\ast f_{\sigma^{\bullet}}=\sigma$ and extending linearly, where $f_{\sigma^{\bullet}}$ is the cochain assigning $1$ to an $i$-cell $\sigma^{\bullet}$ of $\paren{\Z^d}^{\bullet}$ and 0 to all other $i$-cells and we have oriented the cells of $\Z^d$ and $\paren{\Z^d}^{\bullet}$ so that $\ast \delta=\partial \ast.$ Note that this also ensures that $\ast^{-1}\partial=\delta \ast^{-1}.$ Observe that if $P$ is a subcomplex of a cube $\Lambda$ then $P^{\bullet}$ deformation retracts onto  $P^{\ast}=\paren{P^{\bullet}\cap \Lambda^{\bullet}}\cup \partial \Lambda^{\bullet},$ so those spaces have the same (co)homology. 

Alexander duality is stated in terms of the reduced homology and cohomology groups $\tilde{H}_k\paren{P;\Z_q}$ and $\tilde{H}^k\paren{P;\Z_q}.$ These are the same as the usual homology and cohomology groups, except that the dimensions of the $0$-th groups are reduced by one. This is done so that the simplest nonempty topological space, a single point, can have trivial homology in all dimensions despite having one connected component instead of zero. Formally, these groups are defined as follows: Reduced homology is obtained by setting $C_{-1}\paren{P;\Z_q}=\Z_q,$ defining $\partial_0:C_{0}\paren{P;\Z_q}\to C_{-1}\paren{P;\Z_q}$ by setting $\partial_0\paren{v}=1$ for each vertex $v$ of $X$ and extending linearly, and $\tilde{H}_0\paren{P;\Z_q}=\ker \partial_0/\im \partial_1.$ Reduced cohomology is defined analogously, where $\delta^0:\Z_q\to C^0\paren{X;\Z_q}$ sends $t\in \Z_q$ to the constant function that equals $t$ on each vertex of $X.$ We have that $\tilde{H}_k\paren{P;\Z_q}\cong H_k\paren{P;\Z_q}$ and $\tilde{H}^k\paren{P;\Z_q}\cong H^k\paren{P;\Z_q}$ for $k>0$ and $\tilde{H}_0\paren{P;\Z_q}\cong H_0\paren{P;\Z_q}/\Z_q$ and $\tilde{H}^0\paren{P;\Z_q}\cong H^0\paren{P;\Z_q}/\Z_q.$  Importantly, we will use in the proof that $\tilde{H}_k\paren{\Lambda;\Z_q}\cong \tilde{H}^k\paren{\Lambda;\Z_q}\cong 0,$ which fails for $k=0$ for regular homology and cohomology. 

\begin{Theorem}[Alexander Duality]\label{thm:alexander}
Let $P$ be a subcomplex of $\Lambda$ and define the map $\mathcal{I}:\tilde{H}^{d-k-1}\paren{P^{\bullet};\Z_q}\to \tilde{H}_{k}\paren{P;\Z_q}$ by setting $\mathcal{I}\paren{\brac{f}}=\brac{\partial \paren{\ast f}}.$ Then $\mathcal{I}$ is well-defined and is an isomorphism.
\end{Theorem}
\begin{proof}
 For this proof, it will be important that the coboundary operator on $\Z^d$ differs from that on $P^{\bullet}.$ We denote the former simply by $\delta$ and the latter by $\delta_{P^{\bullet}}.$ Note that $\delta_{P^{\bullet}} f=\paren{\delta f}\mid_{P^{\bullet}}.$  

Let $f\in Z^{d-k-1}\paren{P^{\bullet};\Z_q}.$ To show that $\mathcal{I}\paren{\brac{f}}$ is well-defined, we demonstrate that $\partial \ast f$ is a cycle supported on $P$ and that its homology class does not depend on the choice of representative of $\brac{f}.$ Since $\delta_{P^{\bullet}} f=0,$ $\delta f$ is supported on the complement of $P^{\bullet}$ and $\partial \paren{\ast f}=\ast \delta f$ is supported on $P.$ It is a cycle by construction. Moreover, if $g\in {C^{d-k-2}\paren{P^{\bullet};\Z_q}},$ then 
$$\tau \coloneqq \ast\paren{\delta_{P^{\bullet}}g-\delta g}$$
is supported on $P$ and 
$$\mathcal{I}\paren{\delta_{P^{\bullet}} g}=\partial \paren{\ast \delta_{P^{\bullet}}g}= \partial \tau + \partial \paren{\ast \delta g} =\partial \tau$$
so $\brac{\mathcal{I}\paren{\delta_{P^{\bullet}}g}}=0\in H_{k}\paren{P;\Z_q}$ and we can conclude that $\mathcal{I}$ is well-defined. Observe that this fails if we do not use reduced homology and  $k=d-1$: if $f$ is the cochain which assigns $1$ to each vertex of $P^{\bullet}$ then $\brac{f}\neq 0\in H^0\paren{P^{\bullet}}$ but 
$$\partial \paren{\ast f} =\partial \sum_{\sigma\notin P} \sigma=\partial \sum_{\sigma\in \paren{\Z^d}}\sigma-\partial \sum_{\sigma\in P}\sigma=-\partial \sum_{\sigma\in P}\sigma$$
is a boundary, where the sums are taken over $d$-cells $\sigma.$ 

Next, we show that $\mathcal{I}$ is injective.  Suppose that $\mathcal{I}\paren{f}=\partial \tau,$ where $\tau\in C_{k+1}\paren{P;\Z_q}.$ Then $\partial \paren{\ast f-\tau}=0,$ so there exists a $\kappa \in C_{k+2}\paren{\Lambda;\Z_q}$ with $\partial \kappa=\ast f-\tau$ since $H_{k+1}\paren{\Lambda;\Z_q}=0.$  Set $h=\paren{\ast^{-1}}\paren{\kappa}\in C^{d-k-2}\paren{\Lambda^{\bullet};\Z_q}$ and write $h=h'+h''$ where $h'\in C^{d-k-2}\paren{P^{\bullet};\Z_q}$ and $h''$ is supported on cells not contained in $P^{\bullet}.$   Since $P^{\bullet}$ is a cell complex, it is closed under inclusion and the support of $h''$ cannot contain any cells incident to a $\paren{d-k-1}$-cell of $P^{\bullet}.$ As such,
$$\delta h\mid_{P^{\bullet}}=\delta h'\mid_{P^{\bullet}}=\delta_{P^{\bullet}} h'\,.$$
On the other hand,
$$\delta h\mid_{P^{\bullet}}=\paren{\ast^{-1}\partial\kappa}\mid_{P^{\bullet}}=\paren{f-\ast^{-1}\paren{\tau}}\mid_{P^{\bullet}}=f$$
since $\tau$ is supported on $P.$ Thus $\brac{f}=\brac{\delta_{P^{\bullet}} h'}=0$ in $H^{d-k-1}\paren{P^{\bullet};\Z_q}.$

Finally, we demonstrate that $\mathcal{I}$ is surjective. Let $\brac{\tau}\in  \tilde{H}_{k}\paren{P;\Z_q}.$ Since $\tilde{H}_{k}\paren{\Lambda;\Z_q}=0$ there is a $\kappa\in C_{k+1}\paren{\Lambda;\Z_q}$ with $\partial \kappa=\tau.$ Write $\kappa=\kappa'+\kappa''$ where $\kappa'$ is supported on $P$ and $\kappa''$ is supported on cells not contained in $P.$ Since $\brac{\tau}=\brac{\tau-\partial \kappa'}$ in  $\tilde{H}_{k}\paren{P;\Z_q},$ we may replace $\kappa$ with $\kappa''$ if necessary so that $\kappa$ is supported on the cells not contained in $P.$ Set $h=\ast^{-1} \kappa.$ Then $h\in Z^{d-k-1}\paren{P^{\bullet};\Z_q}$ because $\delta h=\ast^{-1}\paren{\tau}$ is supported on cells not contained in $P^{\bullet}.$ Also,
$$\mathcal{I}\paren{h}=\partial \paren{\ast h}=\partial \kappa=\tau.$$ 
\end{proof}

As an immediate consequence of Theorem~\ref{thm:alexander}, we can express the event $V_{\sigma}$ in terms of the cohomology of the dual complex.

\begin{Corollary}\label{cor:dualcomponent}
Let $\sigma$ be an $i$-face of $\Lambda,$ let $P$ be an $i$-dimensional percolation subcomplex of $\Lambda,$ and let $q\geq 2.$ Then $\brac{\partial \sigma}\neq 0 \in H_{i-1}\paren{P;\Z_q}$ if and only if $\brac{f_{\sigma^{\bullet}}}\neq 0 $ in $H^{d-i}\paren{P^{\ast};\Z_q}.$ 
\end{Corollary}

Furthermore, in special cases there is a direct relationship between $V_{\sigma}$ and $V_{\sigma^{\bullet}}.$

\begin{Lemma}\label{lemma:vsigmadual}
Let $P$ be an $i$-dimensional percolation subcomplex of a cube $\Lambda$ and $\sigma$ be an $i$-cell of $\Lambda.$ If $q$ is prime, $i=1,$ or $i=d-1$ then $P^{\bullet}\setminus \sigma^{\bullet}\in V_{\sigma^{\bullet}}^{\mathbf{w}},\sigma^{\bullet}\in P^{\bullet}$ if and only if $P\notin V_{\sigma}.$ 
\end{Lemma}
\begin{proof}
Alexander duality implies that $\beta_{k}\paren{P}=\beta_{d-k-1}\paren{P^{\ast}}\,.$ Combining this with Lemma~\ref{lemma:addone} and the fact that $\paren{P \cup \sigma}^{\bullet} = P^{\bullet} \setminus \sigma^{\bullet}$ yields
\begin{align*}
P\notin V_{\sigma} &\iff \brac{\partial\sigma}\neq 0\in H_{i-1}\paren{P;\Z_q}, \brac{\partial\sigma}= 0\in H_{i-1}\paren{P\cup \sigma;\Z_q}\\
&\iff \beta_{i-1}\paren{P\cup \sigma}=\beta_{i-1}\paren{P}-1, \beta_i\paren{P\cup \sigma}=\beta_i\paren{P}\\
&\iff \beta_{d-i}\paren{P^{\ast}\setminus \sigma^{\bullet}}=\beta_{d-i}\paren{P^{\ast}}-1, \beta_{d-i-1}\paren{P^{\ast}\setminus \sigma^{\bullet}}=\beta_{d-i-1}\paren{P^{\ast}}\\
&\iff \sigma^{\bullet}\in P^{\bullet},P^{\bullet}\setminus \sigma^{\bullet}\in V_{\sigma^{\bullet}}^{\mathbf{w}}\,.
\end{align*}
\end{proof}

For non-prime $q,$ the implication $P \notin V_{\sigma} \implies P^{\bullet}\setminus \sigma^{\bullet}\in V_{\sigma^{\bullet}}^{\mathbf{w}},\sigma^{\bullet}\in P^{\bullet}$ can fail. For example, if $q=4$ and $\brac{\partial\sigma}\neq 0$ but $2\brac{\partial\sigma}=0$ in $\tilde{H}_1\paren{P;\Z_4}$
then $P\notin V_{\sigma}$ and  $P^{\bullet}\setminus \sigma^{\bullet}\notin V_{\sigma^{\bullet}}^{\mathbf{w}}.$ A specific configuration can be found by embedding the example in Figure~1 of~\cite{duncan2025sharp} in a hyperplane in $\Z^4$ and adding a tube between the boundary of a single plaquette $\sigma$ and $\gamma$  which intersects the hyperplane only in $\gamma.$ This subcomplex of $\Z^4$ is homeomorphic to a Klein bottle with a disk removed. Note that there is no subcomplex $P$ of $\Z^3$ with a plaquette $\sigma$ so that  $\brac{\partial\sigma}\neq 0$ but $2\brac{\partial\sigma}=0$ in $\tilde{H}_1\paren{P;\Z_4}$, for that would imply the existence of a dual path of bonds whose linking number with $\partial\sigma$ is $2$ but the non-existence of a path whose linking number with $\partial\sigma$ is $1.$ 

See Figure~\ref{fig:vsigmataudual} for an illustration of the case $i=1.$

\begin{figure}[ht]
\center
\includegraphics[width=0.6\textwidth]{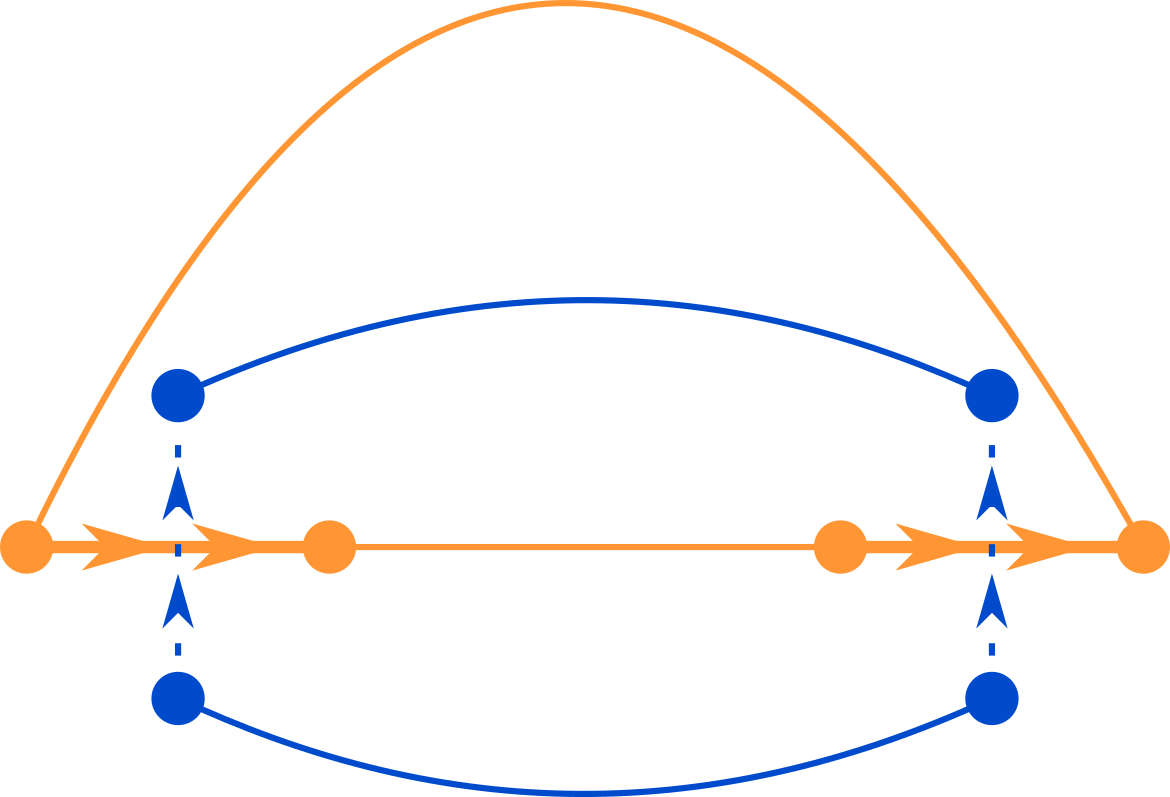}
\caption{When $i=1,$ $q\geq 3,$ and $e_1=\paren{v_1,w_1}$ and $e_2=\paren{v_2,w_2}$ are oriented edges, $V_{e_1,e_2}$ is the event that $v_1$ is connected to $v_2$ and $w_1$ is connected to $w_2$ but that $v_1$ is not connected to $w_1.$ When $q=2,$ $V_{e_1,e_2}$ also allows for an alternate pairing where $v_1$ is connected to $w_2$ and $w_1$ is connected to $v_2$; this implies that $\brac{\partial e_1}=-\brac{\partial e_2}$ for general $q$. An example of the former event is shown in blue, with the dotted blue lines representing $e_1$ and $e_2.$ The dual event $e_1^{\bullet},e_2^{\bullet}\in P^{\bullet}$ and $\brac{\partial e_1^{\bullet}}=-\brac{\partial e_2^{\bullet}}\neq 0$ in $P^{\bullet}\setminus\set{e_1^{\bullet},e_2^{\bullet}}$ with wired boundary conditions is illustrated in orange. The difference in sign is why the dual correlation in Theorem~\ref{thm:massgapdual} pairs $W_{\tau_0}$ with $W_{\tau_N}$ rather than with $W_{\tau_N}^{-1}$.}
\label{fig:vsigmataudual}
\end{figure}

\begin{Lemma}\label{lemma:vsigmataudual}
Fix $q=2,$ and let $P$ be an $i$-dimensional percolation subcomplex of a cube $\Lambda$ and let $\sigma,\tau $ be $i$-cells of $\Lambda.$ 
Then
$$P\in V_{\sigma,\tau}\iff \sigma^{\bullet}\in P^{\bullet},\tau^{\bullet}\in P^{\bullet}, P^{\bullet}\setminus \set{\sigma^{\bullet},\tau^{\bullet}}\in V_{\sigma^{\bullet},\tau^{\bullet}}^{\mathbf{w}}\,.$$
\end{Lemma}
\begin{proof}
Since $q=2,$ the only nontrivial linear dependence between $\brac{\partial \sigma}$ and $\brac{\partial \tau}$ is $\brac{\partial \sigma} = \brac{\partial \tau}$. Thus, the event $V_{\sigma,\tau}$ is equivalent to $P\notin V_{\sigma},P\notin V_{\tau}, P\cup\tau\in V_{\sigma}, P\cup \sigma \in V_{\tau}.$ Applying Lemma~\ref{lemma:vsigmadual} to each of these four events yields the desired result. 
\end{proof}

\section{Topological Formulas}
\label{sec:formulas}
We begin by proving Theorem~\ref{thm:correlation}.

\begin{proof}[Proof of Theorem~\ref{thm:correlation}]
We prove the statement for free boundary conditions, since the proof does not depend on the dimension or boundary conditions.
Using Theorem~\ref{thm:comparisongeneral},
\begin{align*}
\mathbb{E}_{\nu}\paren{W_{\gamma}W_{\gamma'}^{-1}}&=\mathbb{E}_{\nu}\paren{f\paren{\gamma-\gamma'}}\\
&=\mu\paren{\brac{\gamma}=\brac{\gamma'}\in H_i\paren{P;\Z_q}}\\
&=\mu\paren{\brac{\gamma}=\brac{\gamma'}\neq 0}+\mu\paren{\brac{\gamma}=\brac{\gamma'}= 0}\\
&=\mu\paren{V_{\gamma,\gamma'}}+\mu\paren{V_{\gamma}\cap V_{\gamma'}}\,.
\end{align*}
Thus
\begin{align*}
\mathrm{Cov}_{\nu}\paren{W_{\gamma},W_{\gamma'}^{-1}}&=\mathbb{E}_{\nu}\paren{W_{\gamma}W_{\gamma'}^{-1}}-\mathbb{E}_{\nu}\paren{W_{\gamma}}\mathbb{E}_{\nu}\paren{W_{\gamma'}^{-1}}\\
&=\mu\paren{V_{\gamma,\gamma'}}+\mu\paren{V_{\gamma}\cap V_{\gamma'}}-\mu\paren{V_{\gamma}}\mu\paren{V_{\gamma'}}\\
&= \mu\paren{V_{\gamma,\gamma'}}+\mathrm{Cov}_{\mu}\paren{V_{\gamma},V_{\gamma'}}\,.
\end{align*}
\end{proof}

There is an analogue of Theorem~\ref{thm:correlation} involving simpler events in $P$ when the two loops are boundaries of single plaquettes and $q=2.$

\begin{Proposition}\label{prop:correlation_alternate}
Assume the same hypotheses as Theorem~\ref{thm:correlation} but set $q=2$ and let $\sigma,\tau$ be $(i+1)$-plaquettes of $X.$ 
Then    
$$\mathrm{Cov}_{\nu}\paren{W_{\sigma},W_{\tau}^{-1}}= \frac{4}{p^2}\mathrm{Cov}_{\mu}\paren{\sigma\in P,\tau\in P}\,.$$
\end{Proposition}
\begin{proof}
Recalling that $K$ is the coupled measure with both the PRCM and PLGT, we have that
\begin{align*}
\mu\paren{\sigma\in P}&=K\paren{\sigma\in P,W_{\sigma}=1}\\
&=p\nu\paren{W_{\sigma}=1}\\
&=p\mu\paren{V_{\sigma}}+\frac{p}{2}\mu\paren{\neg V_{\sigma}}
\end{align*}
since on the event $\neg V_{\sigma},$ $f\paren{\partial\sigma}$ is distributed uniformly on $\Z\paren{2}$ (see Proposition 28 of~\cite{duncan2025topological}). 

Since $q=2,$ $W_{\sigma}$ and $W_{\tau}$ are independent when conditioned on $P\in\neg V_{\sigma,\tau}.$ Thus
\begin{align*}
\frac{1}{p^2}\mu\paren{\sigma\in P,\tau\in P}&=\frac{1}{p^2}K\paren{\sigma\in P,\tau\in P, W_{\sigma}=W_{\tau}=1}\\
&=\nu\paren{W_{\sigma}=W_{\tau}=1}\\
&=\mu\paren{V_{\sigma},V_{\tau}}+\frac{1}{4}\mu\paren{\neg V_{\sigma},\neg V_{\tau},\neg V_{\sigma,\tau}}\\
&\;\;\;\;+\frac{1}{2}\paren{\mu\paren{V_{\sigma,\tau}}+\mu\paren{V_{\sigma},\neg V_{\tau}}+\mu\paren{\neg V_{\sigma},V_{\tau}}}\\
&=\mu\paren{V_{\sigma},V_{\tau}}+\frac{1}{2}\mu\paren{V_{\sigma},\neg V_{\tau}}+\frac{1}{2}\mu\paren{\neg V_{\sigma},V_{\tau}}\\
&\;\;\;\;+\frac{1}{4}\mu\paren{\neg V_{\sigma},\neg V_{\tau}}+\frac{1}{4}\mu\paren{V_{\sigma,\tau}}\,.\\
\end{align*}

Thus
\begin{align*}
\mathrm{Cov}_{\mu}\paren{\sigma\in P,\tau\in P}&=\mu\paren{\sigma\in P,\tau\in P}-\mu\paren{\sigma\in P}\mu\paren{\tau\in P}\\
&=\frac{p^2}{4}\mu\paren{V_{\sigma,\tau}}+p^2\mathrm{Cov}_{\mu}\paren{V_{\sigma},V_{\tau}}\\
&\;\;\;\;\;+\frac{p^2}{2}\mathrm{Cov}_{\mu}\paren{V_{\sigma},\neg V_{\tau}}+\frac{p^2}{2}\mathrm{Cov}_{\mu}\paren{\neg V_{\sigma},V_{\tau}}\\
&\;\;\;\;\;+\frac{p^2}{4}\mathrm{Cov}_{\mu}\paren{\neg V_{\sigma},\neg V_{\tau}}\\
&=\frac{p^2}{4}\mathrm{Cov}_{\mu}\paren{V_{\sigma},V_{\tau}}+\frac{p^2}{4}\mu\paren{V_{\sigma,\tau}}\\
&=\frac{p^2}{4}\mathrm{Cov}_{\nu}\paren{W_{\sigma},W_{\tau}^{-1}}\,.
\end{align*}
by Theorem~\ref{thm:correlation}.
\end{proof}

\section{Applications}
\label{sec:applications}
We apply the topological formulas from the previous section to prove the main theorems of the paper, starting with Theorem~\ref{thm:massgapdual}. Before doing so, we write out the definition of the correlation length $\xi_{\beta,i}$ for $i>1.$ 
\begin{Definition}
For $\#\in\set{\mathbf{f},\mathbf{w}}$ let $\nu=\nu^{\#}_{\Z^d,\beta,q,d,i}$ and define $\xi_{\beta,i}^{\#}=\xi^{\#}_{\beta,q,d,i}$ via the limit
 $$-\frac{1}{\xi_{\beta,i}^{\#}}=\lim_{N\to\infty} \frac{\log\paren{\mathrm{Cov}_{\nu}\paren{W_{\sigma_1},W_{\sigma_N}^{-1}}}}{N}$$
 where $\sigma_N$ is the $(i+1)$-plaquette  $\brac{0,1}^{i+1}\times\set{N}\times\set{0}^{d-i-2}.$
\end{Definition}
Recall from the introduction that $\hat{\xi}^{\#}_{\beta,q,d,i}$ is defined via the limit
$$-\frac{1}{\hat{\xi}^{\#}_{\beta,q,d,i}}=\lim_{N\to\infty} \frac{\log\paren{\mathrm{Cov}_{\nu}\paren{W_{\tau_1},W_{\tau_N}}}}{N}$$
where $\tau_N=\set{0}^{d-i-1}\times\brac{N,N+1}\times\brac{0,1}^{i}.$ 

\begin{proof}[Proof of Theorem~\ref{thm:massgapdual}]
The case $q=2$ follows immediately from Proposition~\ref{prop:correlation_alternate} and Theorem~\ref{thm:duality}. We consider the case where $q$ is an odd prime in the appendix. 

\end{proof}

Next, we show Theorem~\ref{thm:highlowtemp}, the finiteness of the correlation length at sufficiently high or low temperatures. This is an application of Theorem~\ref{thm:correlation} together with a comparison to Bernoulli plaquette percolation. We say that a set $S$ of $i$-plaquettes is strongly connected if for any plaquettes $\sigma,\sigma' \in S$ there is a tuple $\paren{\sigma = \sigma_1,\sigma_2,\ldots,\sigma_n = \sigma'}$ so that for each $1 \leq j \leq n-1,$ $\sigma_j$ and $\sigma_{j+1}$ share a $\paren{i-1}$-cell. 

Since we will be working with the PRCM rather than PLGT we replace the index $i$ from the statement with $i-1,$ so the PRCM is $i$-dimensional.  The proof is almost identical for free and wired boundary conditions, so we omit the latter case. That is, we show that if $q\geq 2$ and $d\geq 3$ and $\beta$ is sufficiently large or small then  $0<\xi^{\mathbf{f}}_{\beta,q,d,i-1}<\infty.$ We give separate proofs for each for readability. Recall that $\sigma_N$ is the plaquette  $\brac{0,1}^{i}\times\set{N}\times\set{0}^{d-i-1}$ and that $\Lambda_K\paren{\sigma}$ is the box of width $K$ centered at the central point of the plaquette $\sigma.$ 

\begin{proof}[Proof of Theorem~\ref{thm:highlowtemp}, high temperature case]
 Let $p_c^{\mathrm{sc}}\paren{i,d}$ be the critical probability for percolation of strongly connected components of $i$-dimensional independent Bernoulli plaquette percolation. Let $Q = Q_p$ be Bernoulli $i$-plaquette percolation with parameter $p.$ Fix $0 < p' < p_c^{\mathrm{sc}}\paren{i,d}$ and let $\mu = \mu_{p'}.$ For a plaquette $\sigma \in \Z^d,$ let $\mathcal{C}_X\paren{\sigma}$ be its strongly connected component in $X \cup \sigma.$ Notice that the events $V_{\sigma_1}$ and $V_{\sigma_N}$ are measurable with respect to $\mathcal{C}_X\paren{\sigma_1}$ and $\mathcal{C}_X\paren{\sigma_N}$ respectively. By Holley's inequality, the PRCM with parameters $p\in \brac{0,1},q \geq 2$ and arbitrary boundary conditions is stochastically dominated by Bernoulli plaquette percolation with parameter $p$ in any subset of $\Z^d$ (see Lemma 32 of~\cite{duncan2025topological}, which readily generalizes to non-prime $q$). Then since $p' < p_c^{\mathrm{sc}}\paren{i,d}$ (subcritical sharpness of percolation on general quasi-transitive graphs was covered in~\cite{antunovic2008sharpness}), there is a $c = c\paren{p',i} > 0$ so that 
    \[\mu\paren{\abs{\mathcal{C}_{P}\paren{\sigma}}> N} \leq \mathbb{P}_{p'}\paren{\abs{\mathcal{C}_{Q}\paren{\sigma}}> N} \leq \exp\paren{-cN}\,.\]
    Consequently, $\mu\paren{V_{\sigma_1,\sigma_N}} \leq \exp\paren{-cN}.$ 
    Now let 
    \[A_1 = \set{\mathcal{C}_{P}\paren{\sigma_1} \subset \Lambda_{N/4}\paren{\sigma_1}}\] and 
    \[A_N = \set{\mathcal{C}_{P}\paren{\sigma_N} \subset \Lambda_{N/4}\paren{\sigma_N}}\,.\]
    In order to simplify notation, for $\#\in\set{\mathbf{f},\mathbf{w}}$ we will write $\mu^{\#}_1$ for the PRCM in $\Lambda_{N/3}\paren{\sigma_1},$ $\mu_N^{\#}$ for the PRCM in $\Lambda_{N/3}\paren{\sigma_N},$ and $\hat{\mu}_N^{\#}$ for the PRCM on $\Lambda_{N/3}\paren{\sigma_1}\sqcup \Lambda_{N/3}\paren{\sigma_N}.$ Observe that $\hat{\mu}_N^{\#}=\mu^{\#}_1\times \mu_N^{\#}.$

    Let $B_1$ be the event that there is no strongly connected path of plaquettes from $\partial \Lambda_{N/4}\paren{\sigma_1}$ to $\partial \Lambda_{N/3}\paren{\sigma_1}$ and let $B_N$ be the event that there is no strongly connected path of plaquettes from $\partial \Lambda_{N/4}\paren{\sigma_N}$ to $\partial \Lambda_{N/3}\paren{\sigma_N}.$ By a union bound, there is a $c'>0$ so that 
    \begin{align*}
    \mathbb{P}_{p'}\paren{B_1} & = \mathbb{P}_{p'}\paren{B_N} \\
    &\geq 1-\mathbb{P}_{p'}\paren{\bigcup_{\tau \in \partial\Lambda_{N/4}\paren{\sigma_N}}\set{\abs{\mathcal{C}_Q\paren{\tau}} \geq N/24}}\\
    &\geq 1-\abs{\partial\Lambda_{N/4}\paren{\sigma_N}}\exp\paren{-cN/24}\\
    &\geq 1-\exp\paren{-c'N}\,.
    \end{align*}
    Notice that for any increasing event $E$ which is measurable with respect to the states of $\Lambda_{N/4}\paren{\sigma_1},$ we have 
    \begin{equation}\label{eq:wiredfree}
        \mu_1^{\mathbf{w}}\paren{E \cap B_1} \leq \frac{\mu_1^{\mathbf{w}}\paren{E \cap B_1}}{\mu_1^{\mathbf{w}}\paren{B_1}}\leq \mu_1^{\mathbf{f}}\paren{E} \leq \mu_1^{\mathbf{w}}\paren{E}\,,
    \end{equation}
    where the second inequality can be obtained by summing over the possible strongly connected components of $\partial \Lambda_{N/3}\paren{\sigma_1}.$
    Then since free and wired boundary conditions are extremal and the restriction of the measure to two disjoint boxes is stochastically dominated by the product of wired measures on those boxes, we can compute

    \begin{align*}
    \mu\paren{V_{\sigma_1}\cap V_{\sigma_N}}&\leq \hat{\mu}_N^{\mathbf{w}}\paren{V_{\sigma_1}\cap V_{\sigma_N}}\\
    &=\mu_1^{\mathbf{w}}\paren{V_{\sigma_1}} \mu_N^{\mathbf{w}}\paren{V_{\sigma_N}}\\
    &\leq \paren{\mu_1^{\mathbf{w}}\paren{V_{\sigma_1} \cap A_1} + \mu_1^{\mathbf{w}}\paren{\neg A_1}}\paren{\mu_N^{\mathbf{w}}\paren{V_{\sigma_N} \cap A_N} + \mu_N^{\mathbf{w}}\paren{\neg A_N}} \\
    &\leq \paren{\mu_1^{\mathbf{w}}\paren{V_{\sigma_1} \cap A_1}}\mu_N^{\mathbf{w}}\paren{V_{\sigma_N} \cap A_N}+3\mu_1^{\mathbf{w}}\paren{\neg A_1}\\
    \intertext{and}
         \mu\paren{V_{\sigma_1}}\mu\paren{V_{\sigma_N}}&\geq \mu_1^{\mathbf{f}}\paren{P\cap \Lambda_{N/4}\paren{\sigma_1}\in  V_{\sigma_1}}\mu_N^{\mathbf{f}}\paren{P\cap \Lambda_{N/4}\paren{\sigma_N}\in  V_{\sigma_N}}\\
         &\geq \mu_1^{\mathbf{w}}\paren{\paren{P\cap \Lambda_{N/4}\paren{\sigma_1}\in  V_{\sigma_1}}\cap B_1}\mu_N^{\mathbf{w}}\paren{\paren{P\cap \Lambda_{N/4}\paren{\sigma_N}\in  V_{\sigma_N}}\cap B_N}\\
     &\geq \mu_1^{\mathbf{w}}\paren{V_{\sigma_1} \cap A_1 \cap B_1}\mu_N^{\mathbf{w}}\paren{V_{\sigma_N} \cap A_N \cap B_N} \\
     &\geq  \paren{\mu_1^{\mathbf{w}}\paren{V_{\sigma_1}\cap A_1}-\mu_1^{\mathbf{w}}\paren{\neg B_1}}\paren{\mu_N^{\mathbf{w}}\paren{V_{\sigma_N}\cap A_N}-\mu_N^{\mathbf{w}}\paren{\neg B_N}}\\
     &\geq  \paren{\mu_1^{\mathbf{w}}\paren{V_{\sigma_1} \cap A_1}}\mu_N^{\mathbf{w}}\paren{V_{\sigma_N} \cap A_N}-2\mu_1^{\mathbf{w}}\paren{\neg B_1}\,,\\
    \intertext{so}
     \mathrm{Cov}_{\mu}\paren{V_{\sigma_1},V_{\sigma_N}} &\coloneqq \mu\paren{V_{\sigma_1}\cap V_{\sigma_N}} - \mu\paren{V_{\sigma_1}}\mu\paren{V_{\sigma_N}}\\
         &\leq 3\mu_1^{\mathbf{w}}\paren{\neg A_1} + 2\mu_1^{\mathbf{w}}\paren{\neg B_1}\\
        &\leq 3\exp\paren{-cN/4} + 3\exp\paren{-c'N}\,.
        \end{align*}
    
     Combining these bounds with Theorem~\ref{thm:correlation} yields that for PLGT with parameter $\beta\paren{p'},$
    \[\mathrm{Cov}_{\nu}\paren{W_{\sigma_1},W_{\sigma_N}^{-1}} \leq \exp\paren{-cN/4} + 3\exp\paren{-cN/4} + 2\exp\paren{-c'N}\,.\]
 
    There is a trivial exponential lower bound for $\mu\paren{V_{\sigma_1,\sigma_N}}$ found by considering the event that all plaquettes in a minimal set whose boundary is $\partial\sigma_1-\partial\sigma_N$ are open and that all neighboring plaquettes are closed and using the fact that $\mathrm{Cov}_{\mu}\paren{V_{\sigma_1},V_{\sigma_N}}\geq 0$ by FKG. It then follows that $0 < \xi^{\mathbf{f}}_{\beta,q,d,i-1} < \infty$ for $\beta < \beta\paren{p_c^{\mathrm{sc}}\paren{i,d}}.$

 \end{proof}

We derive the low temperature statement from the high temperature dual. 

\begin{proof}[Proof of Theorem~\ref{thm:highlowtemp}, low temperature case] 
    Let $p'$ be such that $0 < \paren{p'}^* < p_c^{\mathrm{sc}}\paren{d-i,d}.$ When $q=2$, the desired statement follows immediately from Proposition~\ref{prop:correlation_alternate}, Lemma~\ref{lemma:vsigmadual}, and the previous argument. We now give a proof that works for $q\geq 3.$
 
    Now let $\mathcal{C}_{Q^{\bullet}}\paren{\sigma_1^{\bullet}}$ be the strongly connected component of $\partial \sigma_1^{\bullet}$ in $P^{\bullet},$ and let
    \[A_1^{\bullet} = \set{\mathcal{C}_{P^{\bullet}}\paren{\sigma_1^{\bullet}} \subset \Lambda_{N/3}\paren{\sigma_1^{\bullet}}}\,.\]
    
    By our choice of $p',$ we have $\mu_{p'}\paren{\abs{\mathcal{C}_{Q^{\bullet}}\paren{\sigma_1^{\bullet}}} > N } \leq \exp\paren{-cN},$ so we now want to relate the events $V_{\sigma_1}$ and $V_{\sigma_1,\sigma_N}$ to $\mathcal{C}_{Q^{\bullet}}\paren{\sigma^{\bullet}}.$ By Corollary~\ref{cor:dualcomponent}, $V_{\sigma}$ occurs if and only if $0 = \brac{f_{\sigma^{\bullet}}} \in H^{d-i}\paren{P^{\bullet}}.$ 

    It follows that $V_{\sigma}$ is measurable with respect to $\mathcal{C}_{P^{\bullet}}\paren{\sigma^{\bullet}}:$  any $g\in C^{d-i-1}\paren{P^{\bullet};\Z_q}$ can be written as $g'+g''$ where $g'$ is supported on cells incident to $\mathcal{C}_{P^{\bullet}}\paren{\sigma^{\bullet}}$ and $g''$ is supported on cells not incident $\mathcal{C}_{P^{\bullet}}\paren{\sigma^{\bullet}},$ and the supports of $\delta g'$ and $\delta g''$ are disjoint so if $\delta g=f_{\sigma^{\bullet}}$ then $\delta g'= f_{\sigma^{\bullet}}$ and $\delta g''=0.$

    Furthermore, if $V_{\sigma_1,\sigma_N}$ occurs then adding $\sigma_N$ to $P$ --- or equivalently removing $\sigma_N^{\bullet}$ from $P^{\bullet}$ --- causes the event $V_{\sigma_1}$ to occur, and it follows that  $\sigma_N^{\bullet} \in \mathcal{C}_{P^{\bullet}}\paren{\sigma_1^{\bullet}}.$ Thus
    
    \[\mu\paren{V_{\sigma_1,\sigma_N}} \leq \mu\paren{\neg A_1^{\bullet}} \leq \exp\paren{-cN/4}\,.\]

    Let $B_1^{\bullet}$ be the event that there is no strongly connected path of dual plaquettes from $\partial \Lambda_{N/4}\paren{\sigma_1^{\bullet}}$ to $\partial \Lambda_{N/3}\paren{\sigma_1^{\bullet}}$ and let $B_N^{\bullet}$ be the event that there is no strongly connected path of dual plaquettes from $\partial \Lambda_{N/4}\paren{\sigma_N^{\bullet}}$ to $\partial \Lambda_{N/3}\paren{\sigma_N^{\bullet}}.$ Since wired and free boundary conditions are dual to each other, we can obtain a dual analogue of (\ref{eq:wiredfree}). Namely, for an increasing event $E$ which is measurable with respect to the states of $\Lambda_{N/4}\paren{\sigma_1^{\bullet}},$ we have
    \begin{equation}\label{eq:wiredfreedual}
        \mu_1^{\mathbf{f}}\paren{E} \leq \mu_1^{\mathbf{w}}\paren{E} \leq \frac{\mu_1^{\mathbf{f}}\paren{E \cap B_1^{\bullet}}}{\mu_1^{\mathbf{f}}\paren{B_1^{\bullet}}}\,.
    \end{equation}

    Then using (\ref{eq:wiredfreedual}), we can now bound
    
    \begin{align*}
    \mu\paren{V_{\sigma_1}\cap V_{\sigma_N}}&\leq \hat{\mu}_N^{\mathbf{w}}\paren{V_{\sigma_1}\cap V_{\sigma_N}}\\
    &=\mu_1^{\mathbf{w}}\paren{V_{\sigma_1}} \mu_N^{\mathbf{w}}\paren{V_{\sigma_N}}\\
    &\leq \paren{\mu_1^{\mathbf{w}}\paren{V_{\sigma_1}^{\mathbf{w}} \cap A_1^{\bullet}} + \mu_1^{\mathbf{w}}\paren{\neg A_1^{\bullet} }}\paren{\mu_N^{\mathbf{w}}\paren{V_{\sigma_N}^{\mathbf{w}} \cap A_N^{\bullet}} + \mu_N^{\mathbf{w}}\paren{\neg A_N^{\bullet} }} \\
    &\leq \mu_1^{\mathbf{w}}\paren{V_{\sigma_1} \cap A_1^{\bullet}}\mu_N^{\mathbf{w}}\paren{V_{\sigma_N} \cap A_N^{\bullet}} + 3\mu_1^{\mathbf{w}}\paren{\neg A_1^{\bullet} }\\
    \intertext{and}
         \mu\paren{V_{\sigma_1}}\mu\paren{V_{\sigma_N}}&\geq \hat{\mu}_N^{\mathbf{f}}\paren{V_{\sigma_1}\cap V_{\sigma_N}}\\
         &\geq \mu_1^{\mathbf{f}}\paren{V_{\sigma_1}}\mu_N^{\mathbf{f}}\paren{V_{\sigma_N}}\\
         &\geq \mu_1^{\mathbf{f}}\paren{V_{\sigma_1} \cap A_1^{\bullet} \cap B_1^{\bullet}}\mu_N^{\mathbf{f}}\paren{V_{\sigma_N} \cap A_N^{\bullet} \cap B_N^{\bullet}}\\
     &\geq \mu_1^{\mathbf{w}}\paren{V_{\sigma_1} \cap A_1^{\bullet} }\mu_1^{\mathbf{f}}\paren{B_1^{\bullet}}\mu_N^{\mathbf{w}}\paren{V_{\sigma_N} \cap A_N^{\bullet}}\mu_N^{\mathbf{f}}\paren{B_N^{\bullet}} \\
     &\geq  \mu_1^{\mathbf{w}}\paren{V_{\sigma_1} \cap A_1^{\bullet} }\mu_N^{\mathbf{w}}\paren{V_{\sigma_N} \cap A_N^{\bullet}} - 2\mu_1^{\mathbf{f}}\paren{\neg B_1^{\bullet}}\,,\\
    \intertext{so}\\
     \mathrm{Cov}_{\mu}\paren{V_{\sigma_1},V_{\sigma_N}} &\coloneqq \mu\paren{V_{\sigma_1}\cap V_{\sigma_N}} - \mu\paren{V_{\sigma_1}}\mu\paren{V_{\sigma_N}}\\
         &\leq 3\mu_1^{\mathbf{w}}\paren{\neg A_1^{\bullet}} + 2\mu_1^{\mathbf{f}}\paren{\neg B_1^{\bullet}}\\
        &\leq 3\exp\paren{-cN/4} + 2\exp\paren{-c'N}\,.
        \end{align*}

    Once again applying Theorem~\ref{thm:correlation} yields that for PLGT with parameter $\beta\paren{p'},$
    \[\mathrm{Cov}_{\nu}\paren{W_{\sigma_1},W_{\sigma_N}^{-1}} \leq \exp\paren{-cN/4} + 3\exp\paren{-cN/4} + 2\exp\paren{-c'N}\,.\]
    As before, the trivial exponential lower bound and FKG then give $0 < \xi^{\mathbf{f}}_{\beta\paren{p'},q,d,i-1} < \infty.$
    
\end{proof}

The next proposition is a stronger version of Theorem~\ref{thm:codimone}. 
\begin{Proposition}\label{prop:covbound}
Let $\nu$ be an infinite volume $(d-2)$-dimensional PLGT measure on $\Z^d$ and  let $g\paren{N} = o\paren{N^{\frac{1}{d-1}}}.$ Set $\gamma = \gamma\paren{N} = \partial\paren{\brac{0,g\paren{N}}^{d-1}\times\set{0}}$ and $\gamma' =\gamma'\paren{N}= \partial\paren{\brac{0,g\paren{N}}^{d-1}\times \set{N}}.$ If $\beta>\beta^*\paren{\beta_c\paren{q}}$ or $q=2$ and $\beta\neq \beta^*\paren{\beta_c\paren{2}}$ then there exist $0<c_0<c_1<\infty$ so that 
$$e^{-c_1 N}\leq \mathrm{Cov}_{\nu}\paren{W_{\gamma\paren{N}},W_{\gamma'\paren{N}}^{-1}}\leq e^{-c_0 N}\,.$$
\end{Proposition}

\begin{proof}

Recall that by Theorem~\ref{thm:correlation}, we have
$$\mathrm{Cov}_{\nu}\paren{W_{\gamma},W_{\gamma'}^{-1}}= \mu\paren{V_{\gamma,\gamma'}}+\mathrm{Cov}_{\mu}\paren{V_{\gamma},V_{\gamma'}}\,.$$
Analogous to the case of $V_{\sigma_1,\sigma_N}$ in the proof of Theorem~\ref{thm:highlowtemp}, there is a trivial exponential lower bound for $\mu\paren{V_{\gamma,\gamma'}}$ , so it suffices to show exponential upper bounds for both terms on the right. We deal with each separately. 

Define $E^{\bullet}\paren{\gamma}$ to be the set of edges of $Q^{\bullet}$ that pass through the rectangle which $\gamma$ bounds and let $\mathcal{C}_{Q^{\bullet}}\paren{\gamma} = \bigcup_{e \in E^{\bullet}\paren{\gamma}} \mathcal{C}_{Q^{\bullet}}\paren{e}.$ Since we are in $\Z^d,$ $V_{\gamma}$ occurs if and only if there is no dual loop whose linking number with $\gamma$ is non-zero modulo $q$ by Proposition 29 of~\cite{duncan2025sharp}. Any such loop must pass through this rectangle, so $V_{\gamma}$ and $V_{\gamma'}$ are measurable with respect to $\mathcal{C}_{Q^{\bullet}}\paren{\gamma}$ and $\mathcal{C}_{Q^{\bullet}}\paren{\gamma'}$ respectively.

Let $E_N$ be the collection of $(d-1)$-cells in $\brac{0,g\paren{N}}^{d-1}\times\set{0}.$ Observe that if $P\in V_{\gamma,\gamma'}$ then $P\cup E_N \in V_{\gamma'}$ so it must be true that  $$\mathcal{C}_{Q^{\bullet}}\paren{\gamma}\cap \mathcal{C}_{Q^{\bullet}}\paren{\gamma'}\neq \varnothing.$$ This is sufficient to show the desired result when the dual bonds are subcritical by the main theorem of~\cite{duminil2019sharp}.   

A little more work is required in the supercritical case to produce a long-range connection between finite components. Assume that $P\in V_{\gamma,\gamma'}$ and let $E_N'$ be the collection of $(d-1)$-cells in $\brac{0,g\paren{N}}^{d-1}\times\set{N}.$ Since $V_{\gamma}$ and $V_{\gamma'}$ are monotone events and $P\cup E_N\in  V_{\gamma}\cap V_{\gamma'},$ $P\cup E_N'\in  V_{\gamma}\cap V_{\gamma'}$ we may find $E\subset \paren{E_N\cup E_N'}\setminus P,$ $\sigma\in E_N\setminus \paren{E\cup P},$ and $\tau\in E_N'\setminus \paren{E\cup P}$ so that $P\cup E\in V_{\gamma,\gamma'}$ but $P\cup E\cup \set{\sigma}, P\cup E\cup \set{\tau}\notin V_{\gamma,\gamma'}.$ 

If we add $\sigma$  and $\tau$ to $P\cup E,$ the addition of the first plaquette reduces the rank of the $(d-2)$-homology by one whereas the second increases the rank of the $(d-1)$-homology by one (where we are using Lemma~\ref{lemma:addone}). If we consider the effect on the dual, we see that removing the first of the two edges of $\sigma^{\bullet},\tau^{\bullet}$ from $P^{\bullet}\setminus E^{\bullet}$ reduces the rank of the $1$-homology and removing the second reduces the number of components by one. That is, if $\sigma^{\bullet}=\langle v,w\rangle,$ $\tau^{\bullet}=\langle v',w'\rangle$ then the removal of both edges separates $v$ from $w$ and $v'$ from $w'$ but only reduces the total number of components by one. As such, by swapping $v'$ and $w'$ if necessary, we must have that $v$ and $v'$ are in the same component of $P^{\bullet}\setminus E^{\bullet}\cup\set{\sigma^{\bullet},\tau^{\bullet}},$ and that $w$ and $w'$ are in a single distinct component of that graph. In summary, if $U_N$ is the event that a vertex of  $\mathcal{C}_{Q^{\bullet}}\paren{\gamma}$ is in a finite component of radius $N$ and $P\in V_{\gamma,\gamma'}$ then the addition of at most $2g\paren{N}^{d-1}$ specific plaquettes to $P$ implies $U_N.$ 

It is known that there is a $c'' = c''\paren{p} > 0$ so that $\mu^{\bullet}_{p}\paren{0\leftrightarrow \partial \Lambda_{N/3},0\not\leftrightarrow \infty} \leq \exp\paren{-c''N}$ when $q=2$ and $p>p_c\paren{2}$~\cite{grimmett2006random,bodineau2005slab} using the fact that the slab threshold coincides with $p_c$ for the latter.  Then, by finite energy and a union bound $$\mu\paren{V_{\gamma,\gamma'}}\leq c^{2g\paren{N}^{d-1}}\mu\paren{U_N} \leq  c^{2g\paren{N}^{d-1}} g\paren{N}^{d-1} e^{-c''N}= e^{-c'' N(1 - o(1))}$$
where $c=q/p.$ 

We now consider the term $\mathrm{Cov}_{\mu}\paren{V_{\gamma},V_{\gamma'}},$ again separately in the two regimes. When $\beta > \beta^*\paren{\beta_c\paren{q}},$ the dual bond set is subcritical and we can apply the same proof as in Theorem~\ref{thm:highlowtemp}. Now let $\beta < \beta^*\paren{\beta_c\paren{2}}$ so the dual bonds are supercritical.  

Recall that as in the proof of Theorem~\ref{thm:highlowtemp}, $V_{\gamma}$ and $V_{\gamma'}$ are measurable with respect to $\mathcal{C}_{Q^{\bullet}}\paren{\gamma}$ and $\mathcal{C}_{Q^{\bullet}}\paren{\gamma'}$ respectively. 

Let $F\paren{\gamma}$ be the event that for each vertex $u$ in an edge $e \in E^{\bullet}\paren{\gamma},$ either $\mathcal{C}_{Q^{\bullet} \setminus E^{\bullet}\paren{\gamma}}\paren{u} \subset \Lambda_{N/4}$ or $\abs{\mathcal{C}_{Q^{\bullet}\setminus E^{\bullet}\paren{\gamma}}\paren{u}} = \infty.$ Similarly, let $F\paren{\gamma'}$ be the event that for each vertex $u$ in an edge $e \in E^{\bullet}\paren{\gamma'},$ either $\mathcal{C}_{Q^{\bullet}\setminus E^{\bullet}\paren{\gamma'}}\paren{u} \subset \Lambda_{N/4}+ N\mathbf{e}_d$ or $\abs{\mathcal{C}_{Q^{\bullet}\setminus E^{\bullet}\paren{\gamma'}}\paren{u}} = \infty.$ Then let $F = F\paren{\gamma} \cap F\paren{\gamma'}.$ Now define $\hat{V}_{\gamma}$ to be the event that $V_{\gamma}$ occurs in $P \cap \Lambda_{N/3}$ and define $\hat{V}_{\gamma'}$ to be the event that $V_{\gamma'}$ occurs in $P \cap \Lambda_{N/3} + N\mathbf{e}_d.$ Notice that the corresponding dual events are in wired boxes. Since the infinite dual component is almost surely unique and one-ended (see Sections 4.3 and 4.4 of~\cite{grimmett2006random}), it follows that $F\cap V_{\gamma} = F \cap \hat{V}_{\gamma}$ and $F\cap V_{\gamma'} = F \cap \hat{V}_{\gamma'}$ almost surely.
Then by Theorem 1.3 of~\cite{duminil2020exponential} there are $c', c''>0$ so that 
\begin{align*}
    \mathrm{Cov}_{\mu}\paren{V_{\gamma},V_{\gamma'}}&= \mu\paren{V_{\gamma}\cap V_{\gamma'}}-\mu\paren{V_{\gamma}}\mu\paren{V_{\gamma'}}\\ 
    &\leq \paren{\mu\paren{V_{\gamma}\cap V_{\gamma'}\cap F}+\mu\paren{\neg F}}-\mu\paren{V_{\gamma}\cap F}\mu\paren{V_{\gamma'}\cap F}\\
    &=\mu\paren{\hat{V}_{\gamma}\cap \hat{V}_{\gamma'}\cap F}-\mu\paren{\hat{V}_{\gamma}\cap F}\mu\paren{\hat{V}_{\gamma'}\cap F}+\mu\paren{\neg F}\\
    &= \mathrm{Cov}_{\mu}\paren{\hat{V}_{\gamma}\cap F ,\hat{V}_{\gamma'}\cap F} + \mu\paren{\neg F}\\
    &\leq \mathrm{Cov}_{\mu}\paren{\hat{V}_{\gamma},\hat{V}_{\gamma'}}+ 3\mu\paren{\neg F} + \paren{\mu\paren{\neg F}}^2\\
    &\leq \exp\paren{-c''N} + 4\exp\paren{-c'N/4}\,,
\end{align*}
yielding the desired exponential bound.
\end{proof}

In particular, Theorem~\ref{thm:codimone} can be recovered by taking $g\paren{N} \equiv 1.$

\section*{Acknowledgements}

We'd like to thank Malin P. Forsstr{\"o}m and Anthony Pizzimenti for interesting discussions, and the anonymous referees for their useful feedback.

We used the LLM Claude Opus 5.5 to check the manuscript for mathematical errors. The LLM suggested the proof of Theorem~\ref{thm:massgapdual} for $q\geq 3$ in the appendix, but all ideas and work in the main portion of the manuscript are due to the authors.

\appendix

\section{Self-Duality via Reflection Positivity}\label{secA1}

In this appendix we prove Theorem~\ref{thm:massgapdual} for the case where $q$ is an odd prime using an argument suggested by AI (Claude Opus 5.5).

The reason that this case is more difficult is that there are multiple possible dependencies among Wilson loop variables that can be created by different relations in the homology of the PRCM.
For example, if there is an $\paren{i+1}$-chain $S$ supported on $P$ with  $\partial S = \partial \sigma + \partial\tau,$ the event $V_{\sigma,\tau}$ need not occur, but the relation $W_{\sigma} = W_{\tau}^{-1}$ holds in the PLGT measure conditional on $P$. We therefore introduce notation to keep track of these quantities.

Let $0\leq i \leq d-2,$ let $\beta>0,$ set $p=1-e^{-\beta},$ let $q$ be prime, let $\#\in\set{\mathbf{f},\mathbf{w}},$ let $\nu=\nu_{\beta,q,i,d}^{\#}$ be PLGT and let $\mu=\mu^{\#}_{\Z^d,1-e^{-\beta},\Z_q,i+1}$ be the corresponding PRCM. Recall that
$$\beta^*\paren{\beta}=\log\paren{\frac{e^{\beta}+q-1}{e^{\beta}-1}},$$
and
$$p^*=\frac{\paren{1-p}q}{\paren{1-p}q + p}\,.$$

For $l \in \Z_q\setminus\set{0},$ let 
\[\mathbf{C}_{\beta}^{\#}\paren{\sigma,\tau,l} \coloneqq \mathrm{Cov}_{\nu}\paren{W_{\sigma},W_{\tau}^{-l}}\]
and denote the energy-energy correlation by
\[\mathbf{E}^{\#}_{\beta}\paren{\sigma,\tau} \coloneqq \mathrm{Cov}_{\nu}\paren{\mathds{1}_{W_{\sigma} = 1},\mathds{1}_{W_{\tau} = 1}}\,.\]

It is not difficult to show that $\mathbf{E}_{\beta}^{\#}\paren{\sigma,\tau}$ satisfies an exact duality relation.

\begin{Proposition}\label{prop:energydual}
Let $\sigma,\tau$ be $(i+1)$-plaquettes. Then 
\[\mathbf{E}^{\mathbf{f}}_{\beta}\paren{\sigma,\tau} = \paren{\frac{p^*}{p}}^2\mathbf{E}^{\mathbf{w}}_{\beta^*}\paren{\sigma^{\bullet},\tau^{\bullet}}\]
where the energy-energy correlation on the left is taken for $\nu_{\beta,q,i,d}^{\bf{f}}$ and the one on the right is taken for the dual $\nu_{\beta^*,q,d-i-2,d}^{\bf{w}}.$ 
\end{Proposition}

\begin{proof}
As observed in the proof of Proposition~\ref{prop:correlation_alternate},
$$\mu\paren{\sigma\in P}=p\nu\paren{W_{\sigma}=1}\,.$$
Similarly,
    \[\mu\paren{\sigma,\tau \in P} = p^2\nu\paren{W_{\sigma} = W_{\tau} = 1}\,.\]
    It then follows that for $p = 1-e^{-\beta},$
    \[\mathbf{E}^{\mathbf{f}}_{\beta}\paren{\sigma,\tau} = \frac{1}{p^2}\mathrm{Cov}_{\mu}\paren{\sigma \in P,\tau \in P}\,\]
    and by the same argument in the dual complex,
    \[\mathbf{E}^{\mathbf{w}}_{\beta^*}\paren{\sigma^{\bullet},\tau^{\bullet}} = \frac{1}{(p^*)^2}\mathrm{Cov}_{\mu^{\bullet}}\paren{\sigma^{\bullet} \in P^{\bullet},\tau^{\bullet} \in P^{\bullet}}\,.\]    
    Then, by taking infinite volume limits of the measures in Theorem~\ref{thm:duality}, we get the desired equality.
\end{proof}

We therefore need only give a comparison between $\mathbf{E}_{\beta}^{\mathbf{f}}\paren{\sigma_1,\sigma_N}$ and $\mathbf{C}_{\beta}^{\mathbf{f}}\paren{\sigma_1,\sigma_N,l}$ which is uniform in $l,$ and likewise for $\mathbf{E}_{\beta^*}^{\mathbf{w}}\paren{\tau_1,\tau_N}$ and $\mathbf{C}_{\beta^*}^{\mathbf{w}}\paren{\tau_1,\tau_N,l}.$ The first bound works equally for both cases.

\begin{Proposition}\label{prop:C<E}
Let $\sigma,\tau$ be $(i+1)$-plaquettes. For every $l \in \Z_q\setminus\set{0},$
    \[0 \leq \mathbf{C}_{\beta}^{\#}\paren{\sigma,\tau,l} \leq \frac{q^2}{q-1}\mathbf{E}_{\beta}^{\#}\paren{\sigma,\tau}\,.\]
\end{Proposition}

\begin{proof}
    If $\#=\mathbf{f}$ let $V_{\sigma,\tau}^l$ be the event that $\brac{\partial \sigma} = l\brac{\partial\tau} \neq 0$ in $H_i\paren{P;\Z_q}.$ Otherwise, if $\#=\mathbf{w},$ let $V_{\sigma,\tau}^l$ be the event that $\brac{\partial \sigma} = l\brac{\partial\tau} \neq 0$ in  $H_i\paren{P\cup \partial \Lambda_N;\Z_q}$ for all sufficiently large $N.$ 

   Then by the same argument as in Theorem~\ref{thm:correlation}, we have 
    \[\mathbf{C}_{\beta}^{\#}\paren{\sigma,\tau,l} = \mu\paren{V_{\sigma,\tau}^{l}} + \mathrm{Cov}_{\mu}\paren{V_{\sigma},V_{\tau}} \geq 0\,,\]
    by positive association (see~\cite{hiraoka2016tutte} for the proof of FKG inequality for the PRCM) where we used that 
    $$\mathbb{E}_{\nu}\paren{W_{\tau}^{-l}}=\mathbb{E}_{\nu}\paren{W_{\tau}}=\mu\paren{V_{\tau}}$$
    as $q$ is prime.
    
    Now since $\mathds{1}_{W_{\sigma} = 1} = \frac{1}{q}\sum_{a \in \Z_q} W_{\sigma}^a,$ $q$ is prime, and the PLGT measure is invariant under automorphisms of $\Z\paren{q},$ we can write
    \begin{equation}\label{eq:EsumC}
        \mathbf{E}_{\beta}^{\#}\paren{\sigma,\tau} = \frac{1}{q^2}\sum_{a,b \in \Z_q\setminus\set{0}}\mathrm{Cov}\paren{W_{\sigma}^a, W_{\tau}^{-b}} = \frac{q-1}{q^2}\sum_{l \in \Z_q\setminus\set{0}} \mathbf{C}_{\beta}^{\#}\paren{\sigma,\tau,l}\,.
    \end{equation}
    Rearranging then gives
    \[\mathbf{C}_{\beta}^{\#}\paren{\sigma,\tau,l} \leq \frac{q^2}{q-1}\mathbf{E}_{\beta}^{\#}\paren{\sigma,\tau}\,.\]
\end{proof}

We now prove a matching lower bound using reflection positivity. First, we show that PLGT is reflection positive, following the proof of~\cite{osterwalder1978gauge,seiler1982gauge}. Let $\theta$ be a reflection of $\R^d$ across a plane $\Pi=\set{x_j=h}$ where $h\in\frac{1}{2}\Z.$ For 
$f\in C^i\paren{\Z^d;\Z\paren{q}}$  define $\theta^*f\in C^i\paren{\Z^d;\Z\paren{q}}$ by $\paren{\theta^*f}\paren{c}=f\paren{\theta c}$ for all oriented $i$-cells $c$, where $\theta c$ carries the orientation induced by $\theta$, and for an observable $F$ set
$\Theta F\paren{f}=\overline{F\paren{\theta^* f}}$. Let $B^+$ be the subcomplex of $\Z^d$ consisting of all cells of $\Z^d$ fully contained in the half-space $\set{x_j\geq h}$.

\begin{Lemma}\label{lemma:RP}
Let $\nu=\nu^{\#}_{\beta,q,i,d}$ for $\#\in\set{\mathbf{f},\mathbf{w}},$ let $\theta$ be a reflection of $\R^d$ across a plane $\Pi=\set{x_j=h},$ and let $F$ be a local observable depending only on the $i$-cells of $B^+.$ If $h\in \Z$ or if $h\in \Z+\frac{1}{2}$ and $F$ is gauge invariant then
$$\mathbb{E}_\nu\brac{F\,\Theta F}\geq 0.$$
\end{Lemma}
\begin{proof}
Let $\Lambda$ be a cube symmetric with respect to $\theta.$ Since the infinite volume measure is a weak limit of finite volume measures on such cubes $\Lambda$ it suffices to prove the statement
for $\nu=\nu^{\#}_{\Lambda,\beta,q,i}.$ The proof is similar for free and wired boundary conditions, so we assume that $\#=\mathbf{f}.$ Let $\Lambda^+=\Lambda\cap \set{x_j\geq h}$ and $\Lambda^-=\Lambda\cap\set{x_j\leq h}.$ Also, let $A^+$ be the collection of cells of $\Lambda^+$ that are not also cells of $\Pi,$ and analogously for $A^-.$

Define the weight associated with a spin $z$ to be $w\paren{z}=e^{\beta\mathds{1}_{z=1}}$, so that the total weight associated to an $i$-cochain $f$ is 
$$\nu\paren{f}\propto W\paren{f}\coloneqq  \prod_{c\in\Lambda^{\paren{i+1}}}w\paren{\delta f\paren{c}}\,.$$

Suppose first that $h\in\Z$. No $(i+1)$-cell has interior points on both sides of $\Pi$,
and $\theta$ fixes the cells in $\Pi$. Write $f=\paren{f_-,f_\Pi,f_+}$ for the restrictions
of $f$ to the $i$-cells of $A^-,$ $\Pi,$ and $A^+.$ Then the
weight factors as 
$$W\paren{f}=m\paren{f_\Pi}B\paren{f_\Pi,f_+}B\paren{f_\Pi,\theta^*f_-}$$
where $m\paren{f_{\Pi}}$ is the product of the weights associated to $(i+1)$-cells contained in $\Pi,$ $B\paren{f_\Pi,f_+}$ is the product of the weights associated to the $\paren{i+1}$-cells in $A^+,$ and  $B\paren{f_\Pi,\theta^*f_-}$ is the product of the weights associated to the $\paren{i+1}$-cells in $A^-.$ 
\[\mathbb{E}_\nu\brac{F\,\Theta F}=\frac{1}{\mathcal{Z}}\sum_{f_\Pi}m\paren{f_\Pi}\abs{\sum_{f_+}
F\paren{f_\Pi,f_+}B\paren{f_\Pi,f_+}}^2\geq 0\]
where 
$$\mathcal{Z}=\sum_{f}W\paren{f}$$
is the partition function of PLGT.

The case where $h\in\Z+\frac12$ is more delicate, since there are both $i$-cells and $\paren{i+1}$-cells that intersect both half-spaces determined by the reflection plane.  To handle this, we use the additional hypothesis that $F$ is gauge invariant. The idea is to apply gauge symmetries near the reflection plane to set $f=1$ on all $i$-cells intersecting it transversely so the weights associated to $(i+1)$-plaquettes passing through $\Pi$ can be written in terms of the values that $f$ assigns to an $i$-plaquette of $A^+$ and its reflection in $A^-.$

Denote by $\mathcal{C}$ the set of $i$-cells of $\Lambda$ that intersect $\Pi$ transversely. Each $c\in\mathcal{C}$ is incident to exactly one $\paren{i-1}$-face $x_c$ contained in
$\Lambda^+$, where $x_c$ is parallel to $\Pi.$  Let $G$ be the group of cochains $g\in C^{i-1}\paren{\Lambda;\Z\paren{q}}$ supported on $\set{x_c:c\in\mathcal{C}}$. Since $x_c$ is the only face of $c$ in $\mathcal{C},$ 
$$\paren{f\cdot\delta g}\paren{c}=f\paren{c}\,g\paren{x_c}^{\pm1}$$
for $g\in G,$  where the sign depends on the orientations of the cells.  It follows that for each $f$ there is a unique $g_f\in G$ with
$f\cdot\delta g_f=1$ on $\mathcal{C}$. Let $S=\set{f:f\mid_{\mathcal{C}}=1}$. The map
$f\mapsto f\cdot\delta g_f$ is $\abs{G}$-to-one onto $S$, and by gauge invariance it preserves
$W$, $F$ and $\Theta F$. Therefore
$$\mathbb{E}_\nu\brac{F\,\Theta F}=\mathbb{E}_\nu\brac{F\,\Theta F\mid S}\,.$$

Let $\mathcal{C}'$ be the set of $(i+1)$-cells that intersect $\Pi$ transversely. Each
$c\in\mathcal{C}'$ has exactly one $i$-face $y_c$ contained in $\Lambda^-$ and one $i$-face $\theta y_c$ contained in $\Lambda^+.$ The remaining $i$-faces of $c$ are in $\mathcal{C}.$  
So on $S$ we have $\delta f\paren{c}=\paren{f\paren{\theta y_c}f\paren{y_c}^{-1}}^{\pm1}$, and since
$w\paren{z}=w\paren{z^{-1}}$,
$$M\paren{f}\coloneqq\prod_{c\in\mathcal{C}'}w\paren{\delta f\paren{c}}
=\prod_{c\in\mathcal{C}'}w\paren{f\paren{\theta y_c}f\paren{y_c}^{-1}}\,.$$
We apply discrete Fourier analysis to $M.$  The characters of $\Z\paren{q}$ are $\chi_a\paren{z}=z^a$ for $a\in\Z_q$, and the discrete Fourier transform of $w$ is
$$\hat{w}\paren{a}=\frac1q\sum_{z\in\Z\paren{q}}w\paren{z}z^{-a}=\frac{e^\beta-1}{q}+\mathds{1}_{a=0}\geq 0\,.$$
Since $\hat w\paren{a}=\hat w\paren{-a}$, for $x,y\in\Z\paren{q}$ it follows that
$w\paren{xy^{-1}}=\sum_a\hat w\paren{a}\,x^a\,\overline{y^a}$, and therefore
$$M\paren{f}=\sum_{\mathbf{a}\in\Z_q^{\mathcal{C}'}}\ \prod_{c\in\mathcal{C}'}\hat w\paren{a_c}
f\paren{\theta y_c}^{a_c}\ \overline{f\paren{y_c}^{a_c}}\,.$$
Write $f_\pm$ for the restriction of $f$ to the $i$-cells of $\Lambda^\pm$, and $B\paren{f_+}$ for
the product of the weights associated to the $(i+1)$-cells of $\Lambda^+$. On $S$ we then have
$W\paren{f}=B\paren{f_+}\,M\paren{f}\,B\paren{\theta^*f_-}$. Since $F$ depends only on $f_+$ and
$\Theta F\paren{f}=\overline{F\paren{\theta^*f_-}}$, substituting $h=\theta^*f_-$ gives
$$\mathbb{E}_\nu\brac{F\,\Theta F\mid S}=\frac{1}{\mathcal{Z}_S}\sum_{\mathbf{a}}\prod_{c\in\mathcal{C}'}
\hat w\paren{a_c}\,\abs{\sum_{f_+}F\paren{f_+}B\paren{f_+}\prod_{c\in\mathcal{C}'}
f\paren{\theta y_c}^{a_c}}^2\geq 0\,,$$
where $\mathcal{Z}_S=\sum_{f\in S}W\paren{f}$.
\end{proof}

We return to showing a matching lower bound for Proposition~\ref{prop:C<E}.
Recall that  $\sigma_N=\brac{0,1}^{i+1}\times\set{N}\times\set{0}^{d-i-2}$ and $\tau_N=\set{0}^{d-i-1}\times\brac{N,N+1}\times\brac{0,1}^{i}.$

\begin{Proposition}\label{prop:C>E}
    For any $N \in \N,$ 
    \[\mathrm{Cov}_{\nu}\paren{W_{\sigma_1},W_{\sigma_N}^{-1}} \geq \frac{q^2}{\paren{q-1}^2}\mathbf{E}_{\beta}^{\#}\paren{\sigma_1,\sigma_N}\]
    and
        \[\mathrm{Cov}_{\nu}\paren{W_{\tau_0},W_{\tau_N}} \geq \frac{q^2}{\paren{q-1}^2}\mathbf{E}_{\beta}^{\#}\paren{\tau_0,\tau_N}\]
\end{Proposition}

\begin{proof}
We begin by establishing the first claim. Let $\theta_N$ be the reflection about the plane $\set{x_{i+2} = (N+1)/2}$ (so that $\theta_N$ sends $\sigma_1$ to $\sigma_N$). Then $\Theta_N W_{\sigma_1} = W_{\sigma_N}^{-1}$ where $\Theta_N F=\overline{F\circ \theta_N}$ for an observable $F.$ 
   
    Consider the matrix $M$ defined by 
    \[M_{a,b} = \mathrm{Cov}\paren{W_{\sigma_1}^a,W_{\sigma_N}^{-b}}\,\]
    for $a,b \in \Z_q\setminus\set{0}.$

    $M$ is positive semidefinite as a consequence of reflection positivity: for
$\vec{v}\in\C^{q-1}$  with coordinates $v_1,\ldots,v_{q-1}$ define
$$F=\sum_{a=1}^{q-1}\overline{v_a}\paren{W^a_{\sigma_1}-\mathbb{E}_\nu W^a_{\sigma_1}}\,.$$
Since $\nu$ is invariant under $\theta_N$, we have
$$\mathbb{E}_\nu\brac{\paren{W^a_{\sigma_1}-\mathbb{E}_\nu W^a_{\sigma_1}}\,
\overline{\paren{W^b_{\sigma_1}-\mathbb{E}_\nu W^b_{\sigma_1}}\circ\theta_N}}=M_{a,b},$$ and thus
$$\vec{v}^{\,*}M\vec{v}=\mathbb{E}_\nu\brac{F\,\Theta_N F}\geq 0$$
by Lemma~\ref{lemma:RP}.
Moreover, since automorphisms of $\Z\paren{q}$ preserve the PLGT measure,
    \[M_{l,l} = M_{1,1} = \mathrm{Cov}\paren{W_{\sigma_1},\paren{W_{\sigma_N}}^{-1}}\,.\]
    Note also that
    \[\sum_{a,b \in \Z_q\setminus\set{0}}M_{a,b} = \paren{q-1}\sum_{l \in \Z_q\setminus\set{0}}\mathbf{C}_{\beta}^{\#}\paren{\sigma_1,\sigma_N,l} = q^2\mathbf{E}_{\beta}^{\#}\paren{\sigma_1,\sigma_N}\,,\]
    where the last equality used (\ref{eq:EsumC}).
    Then we have
    \[\paren{q-1}M_{1,1} = \mathrm{Tr}\paren{M} \geq \lambda_{\mathrm{max}}\paren{M} \geq \frac{\mathbf{1}^T M \mathbf{1}}{\mathbf{1}^T \mathbf{1}} = \frac{q^2\mathbf{E}_{\beta}^{\#}\paren{\sigma_1,\sigma_N}}{q-1}\,,\]
    completing the proof.
    The argument for the second claim concerning $\tau_0$ and $\tau_N$ is identical, except for the fact that $\theta_N$ is replaced by the reflection about $\set{x_{d-i} = (N+1)/2}$ which reverses the orientation of $\tau_0$ and so $\Theta_N W_{\tau_0} = W_{\tau_N}.$ The matrix $M$ is then replaced by 
    \[M_{a,b} = \mathrm{Cov}\paren{W_{\tau_0}^a,W_{\tau_N}^{b}}\,.\]

\end{proof}

\begin{proof}[Proof of Theorem~\ref{thm:massgapdual}]
    Combine Propositions~\ref{prop:energydual},~\ref{prop:C<E}, and~\ref{prop:C>E}. Note that the existence of the limit defining $\xi^{\mathbf{f}}_{\beta,q,d,i}$ implies that the one defining $\hat{\xi}^{\mathbf{w}}_{\beta^*\paren{\beta,q},q,d,d-i-2}$ also exists. 
\end{proof}

\newpage
\bibliographystyle{alpha}
\bibliography{bib}
\end{document}